\documentclass[12pt]{amsart}
\usepackage{amsmath, amssymb, amsfonts}
\usepackage[all]{xypic}
\usepackage[pdftex]{graphicx}
\usepackage[usenames]{color}
\usepackage{upgreek}
\usepackage{mathtools}
\usepackage{stmaryrd}
\usepackage{bbm}

\theoremstyle{plain}
\newtheorem{theorem}{Theorem}[section]
\newtheorem{lemma}[theorem]{Lemma}
\newtheorem{corollary}[theorem]{Corollary}
\newtheorem{prop}[theorem]{Proposition}
\newtheorem{prop-def}[theorem]{Proposition / Definition}

\newtheorem{conj}[theorem]{Conjecture}

\theoremstyle{remark}

\newtheorem{remark}[theorem]{Remark}

\newtheorem{example}[theorem]{Example}

\newtheorem*{note*}{Note}
\newtheorem*{remark*}{Remark}
\newtheorem*{example*}{Example}

\theoremstyle{definition}
\newtheorem*{definition*}{Definition}
\newtheorem*{hypothesis*}{Hypothesis}
\newtheorem*{assumptions*}{Assumptions}
\newtheorem{definition}[theorem]{Definition}

\newcommand{\Z}{\mathbb{Z}}
\newcommand{\R}{\mathbb{R}}
\newcommand{\Q}{\mathbb{Q}}
\newcommand{\C}{\mathbb{C}}
\newcommand{\N}{\mathbb{N}}
\newcommand{\F}{\mathbb{F}}
\newcommand{\Aff}{\mathrm{Aff}}

\newcommand{\Aut}{\mathrm{Aut}}
\newcommand{\Gal}{\mathrm{Gal}}
\newcommand{\Tr}{\mathrm{Tr}}
\newcommand{\GL}{\mathrm{GL}}
\newcommand{\Norm}{\mathrm{Norm}}

\newcommand{\End}{\mathrm{End}}
\newcommand{\Hom}{\mathrm{Hom}}

\newcommand{\res}{\mathrm{res}}
\newcommand{\quot}{\mathrm{quot}}

\newcommand{\tors}{\mathrm{tors}}
\newcommand{\rat}{\mathrm{rat}}
\newcommand{\Irr}{\mathrm{Irr}}
\newcommand{\ind}{\mathrm{ind}}

\newcommand{\ETNC}{\mathrm{ETNC}}

\newcommand{\LTC}{\mathrm{LTC}}
\newcommand{\loc}{\mathrm{loc}}
\newcommand{\Spec}{\mathrm{Spec}}

\newcommand{\id}{\mathrm{id}}

\newcommand{\Leo}{\mathrm{Leo}}
\newcommand{\Char}{\mathrm{Char}}
\newcommand{\Perm}{\mathrm{Perm}}
\newcommand{\Rat}{\mathrm{Rat}}
\newcommand{\odd}{\mathrm{odd}}
\newcommand{\Disc}{\mathrm{Disc}}
\newcommand{\Cl}{\mathrm{Cl}}

\numberwithin{equation}{section}
\usepackage{tabto}

\newcommand{\et}{\mathrm{\acute{e}t}}

\newcommand{\infl}{\mathrm{infl}}

\newcommand{\cyc}{\mathrm{cyc}}

\title[{On the $p$-adic Stark conjecture at $s=1$ and applications}]{On the $p$-adic Stark conjecture at $s=1$ \\ and applications}

\address{
Fachbereich Mathematik\\
Technische Universit\"at Kaiserslautern\\
67663 Kaiserslautern\\
Germany}
\email{thofmann@mathematik.uni-kl.de}
\urladdr{http://www.mathematik.uni-kl.de/$\sim$thofmann}

\author{Henri Johnston}
\address{
Department of Mathematics\\
University of Exeter\\
Exeter\\
EX4 4QF\\
United Kingdom
}
\email{H.Johnston@exeter.ac.uk}
\urladdr{http://emps.exeter.ac.uk/mathematics/staff/hj241}

\author[Andreas Nickel]{Andreas Nickel\\ \\ \tiny{with an appendix by Tommy Hofmann,
		Henri Johnston and Andreas Nickel}}
\address{
Universit\"at Duisburg-Essen\\
Fakult\"at f\"ur Mathematik\\
Thea-Leymann-Stra\ss e 9\\
D-45127 Essen\\
Germany}
\email{andreas.nickel@uni-due.de}
\urladdr{https://www.uni-due.de/$\sim$hm0251/index.html}

\date{Version of 12th October 2019}
\begin{document}

\begin{abstract}
Let $E/F$ be a finite Galois extension of totally real number fields and let $p$ be a prime.
The `$p$-adic Stark conjecture at $s=1$' relates the leading terms at $s=1$ of
$p$-adic Artin $L$-functions to those of the complex Artin $L$-functions attached to $E/F$.
We prove this conjecture unconditionally when $E/\Q$ is abelian. 
We also show that for certain non-abelian extensions $E/F$ the $p$-adic Stark conjecture at $s=1$ is implied by 
Leopoldt's conjecture for $E$ at $p$. Moreover, we prove that for a fixed prime $p$, the $p$-adic Stark conjecture at $s=1$
for $E/F$ implies Stark's conjecture at $s=1$ for $E/F$. 
This leads to a `prime-by-prime' descent theorem for the `equivariant Tamagawa number conjecture' (ETNC) for Tate motives at $s=1$.
As an application of these results, we provide strong new evidence for special cases of the ETNC for Tate motives and the closely related `leading term conjectures'
at $s=0$ and $s=1$.
\end{abstract}

\maketitle

\section{Introduction}\label{sec:introduction}

Let $E/F$ be a finite Galois extension of totally real number fields and let $G=\Gal(E/F)$. Let $p$ be a prime.
In the case that $G$ is abelian, the fundamental work of Deligne and Ribet \cite{MR579702} and of Pierette Cassou-Nogu{\`e}s \cite{MR524276} shows that
one can attach to each irreducible character of $G$ a $p$-adic Artin $L$-function that interpolates values of the corresponding
complex Artin $L$-function at negative integers. This construction can be generalised to the case that $G$ is non-abelian by using Brauer induction.

Roughly speaking, the `$p$-adic Stark conjecture at $s=1$' relates the leading terms at $s=1$ of the complex Artin $L$-function and the $p$-adic Artin $L$-function attached to a character of $G$. For the trivial character it is equivalent to 
the `$p$-adic class number formula' of Colmez \cite{MR922806} together with Leopoldt's conjecture for $F$ at $p$. 
More generally, the leading terms of the two $L$-functions attached to a character of $G$ are related by a certain `comparison period', which is non-zero if Leopoldt's conjecture holds for $E$ at $p$. 

This conjecture is discussed by Tate in \cite{MR782485} where it is attributed to Serre \cite{MR0506177}.
Solomon \cite{MR1906480} noted some slight imprecisions in Tate's discussion and also gave an alternative formulation in the case that 
$G$ is abelian. Burns and Venjakob clarified and simplified Tate's formulation in the general case \cite{MR2290587,MR2749572}, 
and we shall use a slight variant of their version in the present article as it will be the most useful for applications. 
A precise formulation and a more detailed discussion are given in \S \ref{subsec:formulation-p-adic-Stark}.

In the present article, we prove the $p$-adic Stark conjecture at $s=1$ unconditionally when $E/\Q$ is abelian by 
building on work of Ritter and Weiss \cite{MR1423032} and using standard results on Dirichlet $L$-functions and Kubota-Leopoldt $p$-adic $L$-functions.
(Solomon \cite{MR1906480} proved a refinement of his version of the conjecture in this case, but under certain additional hypotheses; see Remark \ref{rmk:solomon-abs-abelian} for further details.)
Moreover, we show that the $p$-adic Stark conjecture at $s=1$ is implied by Leopoldt's conjecture for $E$ at $p$ when every character of $G$
is a virtual permutation character
(in particular, the symmetric groups $S_{n}$ have this property).
By combining the proofs of these two results, 
we also show that if $F=\Q$ and $G \cong \Aff(q) := \mathbb{F}_{q} \rtimes \mathbb{F}_{q}^{\times}$
where $\mathbb{F}_{q}$ is a finite field with $q$ elements and the semidirect product is defined by the natural action, then 
Leopoldt's conjecture for $E$ at $p$ again implies the $p$-adic Stark conjecture at $s=1$ for $E/\Q$.

These results are motivated by their applications to  
the equivariant Tamagawa number conjecture (ETNC) for Tate motives
and the closely related leading term conjectures, both at $s=0$ and at $s=1$.
Building on work of Bloch and Kato \cite{MR1086888}, Fontaine and Perrin-Riou \cite{MR1265546}, and Kato \cite{MR1338860},
Burns and Flach \cite{MR1884523} formulated
the ETNC for any motive over $\Q$ with the action of a semisimple $\Q$-algebra, 
describing the leading term at $s=0$ of an equivariant motivic $L$-function in terms of certain 
cohomological Euler characteristics. 
This is a powerful and unifying formulation which, in particular, 
recovers the Birch and Swinnerton-Dyer conjecture.
We refer the reader to the survey article \cite{MR2088713} for a more detailed overview.

Let $L/K$ be a finite Galois extension of number fields (not necessarily totally real) and let $G=\Gal(L/K)$.
In the case of Tate motives, the ETNC relates certain arithmetic complexes to 
the leading terms at integers of the equivariant complex Artin $L$-function attached to $L/K$.
Burns \cite{MR1863302} formulated the leading term conjecture (LTC) at $s=0$ and he showed that this conjecture  for $L/K$ is equivalent to the ETNC for the pair $(h^{0}(\Spec(L)), \Z[G])$. The advantage of this new formulation is that it is more explicit. Moreover, the LTC at $s=0$ recovers Stark's conjecture at $s=0$ (as interpreted by Tate in \cite{MR782485}),
the `strong Stark conjecture' of Chinburg \cite{MR724009} and Chinburg's `$\Omega(3)$-conjecture' \cite{MR724009,MR786352}. 
It is also equivalent to the `lifted root number conjecture' of Gruenberg, Ritter and Weiss \cite{MR1687551} and implies numerous other conjectures
involving leading terms of Artin $L$-functions at $s=0$ (see \cite[Lecture 3]{MR2882689} for a partial list of such conjectures).
Breuning and Burns \cite{MR2371375} formulated the LTC at $s=1$, which simultaneously refines both Stark's conjecture at $s=1$
(as formulated by Tate in \cite{MR782485}) and Chinburg's `$\Omega(1)$-conjecture' \cite{MR786352}. 
In \cite{MR2804251}, Breuning and Burns also showed that,
under the assumption of Leopoldt's conjecture (at all primes), the LTC at $s=1$ for $L/K$ is equivalent to the
ETNC for the pair $(h^{0}(\Spec(L))(1), \Z[G])$. If the `global epsilon constant conjecture' of Bley and Burns \cite{MR2005875} holds then the 
LTCs at $s=0$ and $s=1$ are equivalent (this is known when $L/K$ is at most tamely ramified, for example).

The main unconditional results to date on the above special cases of the ETNC are as follows
(by the above discussion most of these can also be phrased in terms of the LTCs).
Burns and Greither \cite{MR1992015} showed that if $L/\Q$ is a finite abelian extension and $r \in \Z$ with $r \leq 0$
then the ETNC holds for the pair $\smash{(h^{0}(\Spec(L))(r),\Z[\frac{1}{2}][G])}$;
Flach \cite{MR2863902} resolved important technical difficulties at the prime $p=2$, allowing $\smash{\Z[\frac{1}{2}]}$
in the above to be replaced with $\Z$.
Moreover, Burns and Flach \cite{MR2290586} showed that the analogous results hold when $r>0$.
Bley \cite{MR2226270} showed that if $L$ is a finite abelian extension of an imaginary quadratic field $K$ and if $p$
is an odd prime that splits in $K/\Q$ and does not divide the class number of $K$, then 
the ETNC for the pair $(h^{0}(\Spec(L)), \Z_{(p)}[G])$ holds, where $\Z_{(p)}$ is the localisation of $\Z$ at $p$.
Buckingham \cite{MR3187901} considered certain relative biquadratic extensions.
Now suppose $K=\Q$.
Burns and Flach \cite{MR1981031} showed that $(h^{0}(\Spec(L)),\Z[G])$ holds when $L/\Q$ belongs to a certain explicit infinite family of
$Q_{8}$-extensions of $\Q$, where $Q_{8}$ denotes the quaternion group of order $8$.
 Moreover,
both Navilarekallu \cite{MR2268756} and Janssen \cite{janssen-thesis} computationally verified 
$\smash{(h^{0}(\Spec(L)),\Z[G])}$ for particular $A_{4}$-extensions $L/\Q$, where $A_{4}$ denotes the alternating group of degree $4$.

The present authors \cite{MR3461042} made significant progress for certain finite non-abelian Galois extensions of $\Q$
by introducing `hybrid $p$-adic group rings' and using the functoriality properties of the ETNC to reduce to easier known cases.
In particular, if $p$ is a prime, $m$ is a positive integer and $L/\Q$ is a finite Galois extension with 
$G =\Gal(L/\Q) \cong \Aff(p^{m}) = \F_{p^{m}} \rtimes \F_{p^{m}}^{\times}$ then 
it was shown unconditionally that the ETNC holds for the pairs $\smash{(h^{0}(\Spec(L))(r), \Z[\frac{1}{p}][G])}$
where $r \in \{ 0,1\}$. However, these methods do not extend to give a proof of the ETNC at all primes.

We now return to the situation in which $E/F$ is a finite Galois extension of totally real number fields
and set $G=\Gal(E/F)$. 
Burns \cite{MR3294653} showed 
that under the assumption that certain $\mu$-invariants attached to $E$ and $p$ vanish for all odd prime divisors $p$ of $|G|$,
the $p$-adic Stark conjecture for all characters of $G$ and for all odd primes $p$ implies 
the ETNC for the pair $(h^{0}(\Spec(E))(1), \Z[\frac{1}{2}][G])$.
The proof relies crucially on the descent machinery of Burns and Venjakob \cite{MR2749572}, 
and on the equivariant Iwasawa main conjecture, which has been proven independently by Ritter and Weiss \cite{MR2813337} and by Kakde \cite{MR3091976}, under the assumption that the relevant $\mu$-invariant vanishes.

In the present article, we prove a refinement of the above result that allows us to work prime-by-prime.
In other words, we show that for a fixed odd prime $p$, the vanishing of the relevant $\mu$-invariant attached to $E$ and $p$
together with the $p$-adic Stark conjecture for all characters of $G$ imply the ETNC for the pair 
$(h^{0}(\Spec(E))(1), \Z_{(p)}[G])$.
A key ingredient is a direct proof that for any fixed prime $p$, the $p$-adic Stark conjecture at $s=1$ for all characters of $G$
implies Stark's conjecture at $s=1$ for all such characters. 
(This result is perhaps counter-intuitive because its hypotheses depend on a fixed prime $p$, yet its conclusion does not.)
By combining this prime-by-prime descent result with our new results on the $p$-adic Stark conjecture at $s=1$,
we obtain new evidence for the relevant case of the ETNC.
Thus by tweaking the aforementioned result of Breuning and Burns \cite{MR2804251} to work prime-by-prime, we also obtain results for the LTC at $s=1$. Another ingredient used in some situations is the present authors' aforementioned work on hybrid $p$-adic group rings \cite{MR3461042}.
By using known cases (and by proving a new case) of the global epsilon constant conjecture,
then we obtain results on the LTC at $s=0$ and thus the ETNC for the pair $(h^{0}(\Spec(E)), \Z[G])$.
  
We now give three concrete examples of the new results obtained.
For the first example, let $p$ be an odd prime and let $m$ be a positive integer.
If $E/\Q$ is any totally real Galois extension with
$\Gal(E/\Q) \cong \Aff(p^{m}) = \F_{p^{m}} \rtimes \F_{p^{m}}^{\times}$
then under the assumption that Leopoldt's conjecture for $E$ at $p$ holds, the LTC for $E/\Q$ at $s=1$ holds.
We note that Leopoldt's conjecture for a given number field and prime can be verified computationally 
(see the Appendix).
For the second example, let $C_{n}$ denote the cyclic group of order $n$ and let $G=(C_{3})^{m} \rtimes C_{2}$ where $m$ is a positive integer
and $C_{2}$ acts on $(C_{3})^{m}$ by inversion (in the case $m=1$ we have $G \cong S_{3}$).
If $E/\Q$ is any totally real Galois extension with $\Gal(E/\Q) \cong G$ then 
under the assumption that Leopoldt's conjecture for $E$ at $3$ holds, the LTCs for $E/\Q$ at $s=0$ and $s=1$ both hold.
For the final example, let $G$ be any finite group. 
Then there exist infinitely many Galois extensions of totally real number fields $E/F$ with $\Gal(E/F) \cong G$
such that, if for all odd prime divisors $p$ of $|G|$ both
Leopoldt's conjecture holds for $E$ at $p$ and a certain $\mu$-invariant attached to $E$ and $p$ vanishes, then
the LTCs for $E/F$ at $s=0$ and $s=1$ both hold outside their $2$-primary parts.

In the Appendix (joint with Tommy Hofmann) we use the second result described in the above paragraph 
together with a computational verification of Leopoldt's conjecture at $p=3$ to obtain the 
LTCs at $s=0$ and $s=1$
for all totally real Galois extensions $E/\Q$ with $\Gal(E/\Q) \cong S_{3}$ and $\Disc(\mathcal{O}_{E}) < 10^{20}$ 
(there are $492\, 335$ such extensions). We also obtain the analogous result for all totally real Galois extensions 
$E/\Q$ with $\Gal(E/\Q) \cong D_{12}$ (the dihedral group with $12$ elements) and $\Disc(\mathcal{O}_{E}) < 10^{30}$
(there are $24\, 283$ such extensions).

\subsection*{Acknowledgements}
The first named author acknowledges financial support provided by EPSRC First Grant EP/N005716/1 `Equivariant Conjectures in Arithmetic'.
The second named author acknowledges financial support provided by the DFG 
within the Collaborative Research Center 701
`Spectral Structures and Topological Methods in Mathematics'.
The authors are indebted to Werner Bley for his help in running the \textsc{Magma} \cite{MR1484478} code bundled in Debeerst's PhD thesis \cite{debeerst-thesis}
(this is used in the proof of Theorem \ref{thm:GEGG-for-C3m:C2-extensions}).
The authors are also grateful to Alex Bartel, Tim Dokchitser, Xavier-Fran\c{c}ois Roblot and Otmar Venjakob for helpful conversations and correspondence, and to two anonymous referees for helpful comments and suggestions.

\subsection*{Notation and conventions}
All rings are assumed to have an identity element and all modules are assumed
to be left modules unless otherwise stated. We fix the following notation:

\medskip

\begin{tabular}{ll}
$S_{n}$ & the symmetric group of degree $n$ \\
$A_{n}$ & the alternating group of degree $n$  \\
$C_{n}$ & the cyclic group of order $n$ \\
$D_{2n}$ & the dihedral group of order $2n$ \\
$V_{4}$ & the subgroup of $A_{4}$ generated by double transpositions\\
$\F_{q}$ & the finite field with $q$ elements, where $q$ is a prime power\\
$\Aff(q)$ & the affine group isomorphic to $\F_{q} \rtimes \F_{q}^{\times}$ defined in \S \ref{subsec:affine}\\
$\zeta_{n}$ & a primitive $n$-th root of unity\\
$\Tr_{L/K}$ & the trace map for the field extension $L/K$\\
$Quot(R)$ & the field of fractions of an integral domain $R$\\
$M_{n}(R)$ & the ring of $n \times n$ matrices over a ring $R$\\
$\Sigma_{\infty}(K)$ & the set of infinite places of a number field $K$\\
$\Sigma_{p}(K)$ & the set of places of a number field $K$ above a rational prime $p$\\
\end{tabular}

\medskip
\noindent
A finite Galois extension of totally real number fields will usually be denoted by $E/F$.
By contrast, $L/K$ will usually denote a finite Galois extension of number fields, neither of which is necessarily totally real.

\section{Algebraic Preliminaries}

\subsection{Representations and characters of finite groups}\label{subsec:reps-and-chars-of-finite-groups}
Let $G$ be a finite group and let $K$ be a field of characteristic $0$. 
We write $R_{K}^{+}(G)$ for the set of characters attached to finite-dimensional $K$-valued representations of $G$,
and $R_{K}(G)$ for the ring of virtual characters generated by $R_{K}^{+}(G)$.  
Moreover, we let $\Irr_{K}(G)$ denote the subset of irreducible characters in $R_{K}^{+}(G)$ 
and let $\Char_{K}(G)$ denote the ring of $K$-valued virtual characters of $G$.
Thus we have containments
\[
\Irr_{K}(G) \subset R_{K}^{+}(G) \subset R_{K}(G) \subset \Char_{K}(G).
\]

We let $\mathbbm{1}_{G}$ denote the trivial character of $G$ and for $\chi, \psi \in R_{K}(G)$ we write $\langle \psi, \chi \rangle_{G}$
for the usual inner product of virtual characters.
For a subgroup $H$ of $G$ and $\psi \in R_{K}^{+}(H)$ we write $\ind_{H}^{G} \psi \in R_{K}^{+}(G)$ for the induced character;
for a normal subgroup $N$ of $G$ and $\chi \in R_{K}^{+}(G/N)$ we write $\infl_{G/N}^{G} \chi \in R_{K}^{+}(G)$ for the inflated character.
For $\sigma \in \Aut(K)$ and $\chi \in \Char_{K}(G)$ we set 
$\chi^{\sigma} := \sigma \circ \chi$ and note that this defines a group action
from the left even though we write exponents on the right of $\chi$.

We write $\Perm(G)$ for the ring of characters of virtual permutation representations of $G$, that is, $\Z$-linear combinations of characters of the form $\ind^{G}_{H} \mathbbm{1}_{H}$ where $H$ ranges over subgroups of $G$. It is important to note that each of the inclusions
\[
\Perm(G) \subset R_{\Q}(G) \subset \Char_{\Q}(G)
\]
may be strict. 

\subsection{Endomorphisms of modules over group algebras}\label{subsec:endomorphisms}
Let $G$ be a finite group and let $K$ be a field of characteristic $0$. 
For any $\chi \in R_{K}^{+}(G)$ we fix a $K[G]$-module $V_{\chi}$ with character $\chi$.
For any $K[G]$-module $M$ with $\dim_{K} M < \infty$ and any $\alpha \in \End_{K[G]}(M)$ we write
$M^{\chi}$ for the $K$-vector space 
\[
\Hom_{K[G]}(V_{\chi}, M) 
\cong \Hom_{K}(V_{\chi}, M)^{G} 
\] 
and $\alpha^{\chi}$ for the induced map $(f \mapsto \alpha \circ f) \in \End_{K}(M^{\chi})$.
We note that $\det_{K}(\alpha^{\chi})$ is independent of the choice of $V_{\chi}$.
The following is similar to \cite[Chapitre I, 6.4]{MR782485}.

\begin{lemma}\label{lem:properties-of-rho-parts-functor}
Let $M$ be a $K[G]$-module with $\dim_{K} M < \infty$ and let $\alpha \in \End_{K[G]}(M)$.
Let $H$ be a subgroup of $G$ and let $M |_{H}$ denote $M$ considered as a $K[H]$-module.
Let $N$ be a normal subgroup of $G$ and let $M^{N}$ denote the $K[G/N]$-module of $N$-invariants of $M$.
\begin{enumerate}
\item If $\chi_{1},\chi_{2} \in R_{K}^{+}(G)$ then  
${\det}_{K}(\alpha^{\chi_{1} + \chi_{2}}) = {\det}_{K}(\alpha^{\chi_{1}})  {\det}_{K}(\alpha^{\chi_{2}})$.
\item If $\chi \in R^{+}_{K}(H)$ then $M^{\ind^{G}_{H} \chi} \cong (M |_{H})^{\chi}$ and 
${\det}_{K}(\alpha^{\ind^{G}_{H} \chi}) = {\det}_{K}(\alpha^{\chi})$.
\item If  $\chi \in R^{+}_{K}(G/N)$ then $M^{\infl^{G}_{G/N} \chi} \cong (M^{N})^{\chi}$ and 
${\det}_{K}(\alpha^{\infl^{G}_{G/N} \chi}) = {\det}_{K}( (\alpha |_{M^{N}})^{\chi})$.
\end{enumerate} 
\end{lemma}

\begin{proof}
In the appropriate bases, the matrix for $\alpha^{\chi_{1}+\chi_{2}}$ is a block matrix whose blocks are
the matrices for $\alpha^{\chi_{1}}$ and  $\alpha^{\chi_{2}}$, and this gives claim (i).
Claim (ii) follows from Frobenius reciprocity, i.e., the natural isomorphism 
$\Hom_{K[G]}(\ind_{H}^{G} V_{\chi}, M) \cong \Hom_{K[H]}(V_{\chi}, M |_{H})$.
Similarly, the natural isomorphism 
$\Hom_{K[G]}(\infl_{G/N}^{G} V_{\chi}, M) \cong \Hom_{K[G/N]}(V_{\chi}, M^{N})$ gives claim (iii).
\end{proof}

\section{Stark's conjecture at $s=1$}

\subsection{Artin $L$-functions}
Let $L/K$ be a finite Galois extension of number fields and let $G=\Gal(L/K)$.
Let $\Sigma$ be a finite set of places of $K$ containing the set of infinite places $\Sigma_{\infty}(K)$.
For each character $\chi \in R_{\C}(G)$ we write $L_{\Sigma}(s,\chi)$ for the $\Sigma$-truncated 
(complex) Artin $L$-function attached to $\chi$ (see \cite[Chapter VII, \S 10]{MR1697859}). 
We recall that $L_{\Sigma}(s,\chi_{1}+\chi_{2})=L_{\Sigma}(s,\chi_{1})L_{\Sigma}(s,\chi_{2})$
and that $L_{\Sigma}(s,\chi)$ is invariant under induction and inflation of characters.
Moreover, $L_{\Sigma}(s,\mathbbm{1}_{G}) = \zeta_{K,\Sigma}(s)$, the $\Sigma$-truncated Dedekind zeta-function attached to $K$,
which has a simple pole at $s=1$. In fact, writing $L_{\Sigma}^{*}(1,\chi)$ for the leading term at $s=1$ of $L_{\Sigma}(s,\chi)$,
it is well-known that 
\begin{equation}\label{eq:Artin-L-order-of-vanishing}
L_{\Sigma}^{*}(1,\chi) = \lim_{s \rightarrow 1} (s-1)^{\langle \mathbbm{1}_{G}, \chi \rangle_{G}} \cdot L_{\Sigma}(s,\chi)
\end{equation}
(see the discussion of \cite[p.\ 225]{MR0218327}, for instance).
Here the crucial point is the correct order of vanishing.

\subsection{Stark's conjecture at $s=1$}
For any place $w$ of $L$ we write $L_{w}$ for the completion of $L$ at $w$.
Let $L_{\infty} := \prod_{w \in \Sigma_{\infty}(L)} L_{w}$. 
Define $\Tr_{\infty} : L_{\infty} \rightarrow \R$ by $(x_{w})_{w \in \Sigma_{\infty}(L)} \mapsto \sum_{w \in \Sigma_{\infty}(L)} \Tr_{L_{w}/\R}(x_{w})$
and denote the kernels of the trace maps by $L^{0}_{\infty} := \ker(\Tr_{\infty})$ and $L^{0} := \ker(\Tr_{L/\Q})$.
Then we have a commutative diagram of $\R[G]$-modules with exact rows
\[
\xymatrix{
0  \ar[rr] & & {\R \otimes_{\Q} L^{0}} \ar[rr]^{\subset} \ar[d]^{\mu_{L}} & &
{\R \otimes_{\Q} L} \ar[rr]^{\quad \R \otimes_{\Q} \Tr_{L/\Q}} \ar[d]^{\mu_{L}'} & & \R \ar@{=}[d] \ar[rr] & & 0\\
0  \ar[rr] & & {L^{0}_{\infty}} \ar[rr]^{\subset} & &
L_{\infty} \ar[rr]^{\Tr_{\infty}} & & \R \ar[rr] & & 0
} 
\]
where $\mu_{L}$ is the restriction of the $\R[G]$-module isomorphism $\mu_{L}': \R \otimes_{\Q} L \xrightarrow{\, \sim \,}  L_{\infty}$
given by $x \otimes y \mapsto (x \sigma_{w}(y))_{w \in \Sigma_{\infty}(L)}$ and $\sigma_{w}: L \rightarrow L_{w}$ is an embedding for each infinite place $w$. 
Let $\exp_{\infty} : L_{\infty} \rightarrow  L_{\infty}^{\times}$ denote the product of exponential maps and
let $\Delta_{\infty} : L^{\times} \rightarrow L_{\infty}^{\times}$ denote the diagonal embedding. 
Set $\log_{\infty}(\mathcal{O}_{L}^{\times}) :=  \{ x \in L_{\infty} : \exp_{\infty}(x) \in \Delta_{\infty}(\mathcal{O}_{L}^{\times}) \}$.
Then one can use the proof of the Dirichlet unit theorem to show that $\log_{\infty}(\mathcal{O}_{L}^{\times})$
is a full lattice in $L^{0}_{\infty} \cong \R \otimes_{\Q} L^{0}$ (see \cite[Lemma 6.3]{debeerst-thesis}), and so 
there is an isomorphism of $\C[G]$-modules 
\begin{equation}\label{eqn:mu-infty}
\mu_{\infty} : \C \otimes_{\Z} \log_{\infty}(\mathcal{O}_{L}^{\times}) \xrightarrow{\, \sim \,} \C \otimes_{\Q} L^{0}. 
\end{equation}
(Note that $\mu_{L}'$, $\mu_{L}$ and $\mu_{\infty}$ are canonical when $L$ is totally real; otherwise they depend of the choice of embeddings $\sigma_{w}: L \rightarrow L_{w}$ for complex places $w$.)
Hence there exists a (non-canonical) $\Q[G]$-isomorphism
$g: L^{0} \xrightarrow{\, \sim \,} \Q \otimes_{\Z} \log_{\infty}(\mathcal{O}_{L}^{\times})$
(use \cite[Lemma 8.7.1]{MR2392026}, for instance).
For any such $g$ and any $\chi \in R_{\C}^{+}(G)$ we define
\[
R_{1}(\chi,g) := {\det}_{\C} (\mu_{\infty} \circ (\C \otimes_{\Q} g))^{\chi} \in \C^{\times}.
\]

\begin{conj}[Stark's conjecture at $s=1$]\label{conj:Stark-at-1}
Let $L/K$ be a finite Galois extension of number fields and let $G=\Gal(L/K)$.
Let $\Sigma$ be a finite set of places of $K$ containing $\Sigma_{\infty}(K)$.
Let $\chi \in R_{\C}^{+}(G)$. Then for every $\Q[G]$-isomorphism
$g: L^{0} \xrightarrow{\, \sim \,} \Q \otimes_{\Z} \log_{\infty}(\mathcal{O}_{L}^{\times})$ and every $\sigma \in \Aut(\C)$
we have
\[
	\frac{R_{1}(\chi^{\sigma},g)}{L_{\Sigma}^{\ast}(1, \chi^{\sigma})} = 
	\sigma \left(\frac{R_{1}(\chi,g)}{L_{\Sigma}^{\ast}(1, \chi)}\right).
\]
\end{conj}

\begin{remark}\label{rmk:reformulation-stark-1}
Conjecture \ref{conj:Stark-at-1} is in fact a reformulation of Stark's conjecture at $s=1$ as stated in
{\cite[Chapitre I, Conjecture 8.2]{MR782485}}; the two formulations are indeed equivalent as explained 
in the second paragraph of the proof of \cite[Proposition 3.6]{MR2371375}.
Moreover, by {\cite[Chapitre I, Th\'eor\`eme 8.4]{MR782485}} 
Stark's conjecture at $s=1$ is equivalent to Stark's conjecture at $s=0$ ({\cite[Chapitre I, Conjecture 5.4]{MR782485}}). 
\end{remark}

\begin{remark}\label{rmk:stark-1-independent-of-choices}
Stark's conjecture at $s=1$ is independent of certain choices as follows. (i) By {\cite[Chapitre I, Proposition 8.3]{MR782485}}, if it is true for some choice of $g$ then it is true for every choice of $g$. (ii) By considering Euler factors one can show that if it is true for some choice of $\Sigma$ then it is true for every choice of $\Sigma$.
(iii) A straightforward substitution shows that if it is true for $\chi$ then it is true for $\chi^{\tau}$ for every choice of $\tau \in \Aut(\C)$.
\end{remark}

\begin{remark}\label{rmk:Stark-s=01-for-rational-valued-characters}
The following facts are proven in \cite[Chapitre II]{MR782485}.
Using the analytic class number formula, one can show that Stark's conjectures at $s=0$ and $s=1$ hold for the trivial character $\mathbbm{1}_{G}$.
Moreover, the truth of these conjectures is invariant under induction and respects addition of characters, and from this it is straightforward 
to deduce that they hold for all $\chi \in \Perm(G)$. 
With more effort, one can show that in fact they hold for all $\chi \in \Char_{\Q}(G)$
(recall that the inclusion $\Perm(G) \subset \Char_{\Q}(G)$ can be strict).
\end{remark}

\section{The $p$-adic Stark conjecture at $s=1$}

\subsection{Leopoldt's conjecture}
For a comprehensive discussion of Leopoldt's conjecture, we refer the reader to \cite[Chapter X, \S 3]{MR2392026}.
Let $K$ be a number field.
For a finite place $w$ of $K$, let $U_{K_{w}}$ denote the group of units of the completion $K_{w}$ and let $U_{K_{w}}^{1}$ denote the subgroup of principal units.
Let $p$ be a prime and let $\Sigma_{p}(K)$ denote the set of places of $K$ above $p$. 
After taking $p$-adic completions of abelian groups, the diagonal embedding 
$\mathcal{O}_{K}^{\times} \longrightarrow  \prod_{w \in \Sigma_{p}(K)} U_{K_{w}}$
gives rise to a canonical homomorphism 
\begin{equation*}\label{eq:lambda_p}
\lambda_{p}: \Z_{p} \otimes_{\Z} \mathcal{O}_{K}^{\times} \longrightarrow \textstyle{\prod_{w \in \Sigma_{p}(K)}} U_{K_{w}}^{1}. 
\end{equation*}
We say that $\Leo(K,p)$ holds when $\lambda_{p}$ is injective and we take our formulation of Leopoldt's conjecture for $K$ at $p$ to be that of \cite[(10.3.5)]{MR2392026}.
Moreover, we write $R_{F,p}$ for the $p$-adic regulator of a totally real number field $F$ (see \cite[(10.3.4)]{MR2392026}).
We now recall the following results that we shall use throughout this article, often without further reference.

\begin{theorem}\label{thm:leo-results}
Let $p$ be prime and let $K$ be a number field.
\begin{enumerate}
\item Leopoldt's conjecture for $K$ at $p$ holds if and only if $\Leo(K,p)$ holds.
\item The homomorphism $\Q_{p} \otimes_{\Z_{p}} \lambda_{p}$ is injective if and only if $\Leo(K,p)$ holds.
\item If $K/\Q$ is a finite abelian extension then $\Leo(K,p)$ holds.
\item If $M$ is a subfield of $K$ then $\Leo(K,p)$ implies $\Leo(M,p)$.
\item If $F$ is a totally real number field, then $R_{F,p} \neq 0$ if and only if $\Leo(F,p)$ holds.
\end{enumerate}
\end{theorem}

\begin{proof}
The map $\lambda_{p}$ is always injective on the $p$-torsion part of  $\Z_{p} \otimes_{\Z} \mathcal{O}_{K}^{\times}$
and thus the injectivity of $\lambda_{p}$ is equivalent to the injectivity of $\Q_{p} \otimes_{\Z_{p}} \lambda_{p}$, establishing assertion (ii).
Hence assertion (i) follows from \cite[Theorem 10.3.6 (iii)]{MR2392026} after observing that (using the notation of ibid.)
$\smash{\hat{U}_{\mathfrak{p}}} \cong U_{\mathfrak{p}}^{1}$ for $\smash{\mathfrak{p} \in \Sigma_{p}(K)}$
and $\smash{\hat{U}_{\mathfrak{p}}}$ is finite for $\smash{\mathfrak{p} \notin \Sigma_{p}(K)}$.
Ax \cite{MR0181630} reduced assertion (iii) to a $p$-adic version of Baker's theorem, which was  
proved by Brumer \cite{MR0220694} (also see \cite[Theorem 10.3.16]{MR2392026}).
For assertions (iv) and (v), see \cite[Theorem 10.3.11]{MR2392026} and \cite[p.\ 627]{MR2392026}, respectively.
\end{proof}

\subsection{The $p$-adic analytic class number formula}\label{subsec:p-adic-ACNF}
We follow the exposition of \cite[Chapter XI, \S 6.2]{MR2392026}, to which we refer the reader for further details and references.
Let $F$ be a totally real number field and let $\zeta_{F}(s)$ denote the Dedekind zeta function attached to $F$.
Then by the Siegel-Klingen theorem $\zeta_{F}(1-n) \in \Q$ for integers $n \geq 1$ and these values are non-zero when $n$ is even;
moreover, if $n>1$ is odd then $\zeta_{F}(1-n)=0$.
Furthermore, for each prime $p$ there exists a unique $p$-adic analytic function 
$\zeta_{F,p} : \Z_{p} \backslash \{ 1 \} \rightarrow \Q_{p}$ satisfying
\[
\zeta_{F,p}(1-n) = \zeta_{F}(1-n) \prod_{v \in \Sigma_{p}(F)} (1 - \Norm_{F/\Q}(v)^{n-1}) 
\]
for all $n>1$ with $d \mid n$ where $d=[F(\zeta_{2p}):F]$.
The function $\zeta_{F,p}(s)$ is called the $p$-adic zeta function attached to $F$ and has at most a simple pole at $s=1$.
Colmez proved the following result analogous to the usual analytic class number formula at $s=1$.

\begin{theorem}[{\cite[\S 5]{MR922806}}]\label{thm:p-adic-class-number-formula}
Let $p$ be a prime and let $F$ be a totally real number field. Then
\begin{equation*}\label{eq:p-adic-class-number-formula}
\lim_{s \rightarrow 1} (s-1) \zeta_{F,p}(s) = \frac{2^{[F:\Q]} h_{F} R_{F,p}}{w_{F}\sqrt{|d_{F}|}} \prod_{v \in \Sigma_{p}(F)} (1 - \Norm_{F/\Q}(v)^{-1}) 
\end{equation*}
where $h_{F}$, $R_{F,p}$ and $d_{F}$ are the class number, $p$-adic regulator and discriminant of $F$, respectively, 
and $w_{F}(=2)$ is the number of roots of unity contained in $F$.
\end{theorem}

\begin{remark} \label{rmk:sign-of-quotient}
The quantities $R_{F,p}$ and $\sqrt{|d_{F}|}$ are only defined up to sign.
However, their quotient can be well-defined as follows (we recall the explanation of \cite[\S 2.3]{MR0337898}).
Let $\log_{p} : \C_{p}^{\times} \rightarrow \C_{p}$ denote the $p$-adic logarithm.
Let $n=[F:\Q]$ and let $\sigma_{1}, \dots, \sigma_{n}$ be the embeddings of $F$ into $\R$.
Let $\varepsilon_{1}, \dots, \varepsilon_{n-1}$ be a system of fundamental units in $\mathcal{O}_{F}^{\times}$ and let 
$\omega_{1}, \dots, \omega_{n}$ be a $\Z$-basis of $\mathcal{O}_{F}$. 
The usual regulator $R_{F,\infty} := \det \left( \log |\sigma_{i}(\varepsilon_{j})| \right)_{1 \leq i,j \leq n-1}$, the $p$-adic regulator
$R_{F,p} := \det ( \log_{p} \iota \circ \sigma_{i}(\varepsilon_{j}))_{1 \leq i,j \leq n-1}$ and
$\sqrt{|d_{F}|} := \det \left(\sigma_{i}(\omega_{k}) \right)_{1 \leq i,k \leq n}$ are all well-defined up to sign. 
Moreover, one can make choices of the tuples $(\sigma_{i}),(\varepsilon_{j})$ and $(\omega_{k})$
such that the quotient $R_{F,\infty} / \sqrt{|d_{F}|}$ is positive. 
Note that it is this quotient that appears in the usual analytic class number formula at $s=1$.
With the same choices, for any field isomorphism $\iota : \C \cong \C_{p}$ the quotient
\[
\frac{R_{F,p}}{\sqrt{|d_{F}|}} 
:= \frac{\det \left( \log_{p} \iota \circ \sigma_{i}(\varepsilon_{j}) \right)_{1 \leq i,j \leq n-1}}{\det \left(\iota \circ \sigma_{i}(\omega_{k}) \right)_{1 \leq i,k \leq n}}
= \frac{\det \left( \log_{p} \iota \circ \sigma_{i}(\varepsilon_{j}) \right)_{1 \leq i,j \leq n-1}}{\iota (\sqrt{|d_{F}|})}
\]
is well-defined and does not depend on the choice of $\iota$.
\end{remark}

\begin{corollary}\label{cor:simple-pole-iff-leopoldt}
$\zeta_{F,p}(s)$ has a (simple) pole at $s=1$ if and only if $\Leo(F,p)$ holds.
\end{corollary}

\subsection{A certain `comparison period'}\label{subsec:comparison-period}
Let $E/F$ be a finite Galois extension of totally real number fields and let $G=\Gal(E/F)$.
Recall from \eqref{eqn:mu-infty} that
\[
\mu_{\infty}: \C \otimes_{\Z} \log_{\infty}(\mathcal{O}_{E}^{\times}) \xrightarrow{\, \sim \,} \C \otimes_{\Q} E^{0}
\]
is a canonical isomorphism of $\C[G]$-modules. 
Let $\exp_{\infty}': \log_{\infty}(\mathcal{O}_{E}^{\times}) \rightarrow \mathcal{O}_{E}^{\times}$ 
denote the map induced by $\exp_{\infty}: E_{\infty} \rightarrow E_{\infty}^{\times}$ and the inverse of the restriction of the diagonal embedding 
$\Delta_{\infty}:E^{\times} \rightarrow E_{\infty}^{\times}$ to $\mathcal{O}_{E}^{\times}$. Note that $\exp'_{\infty}$ is injective since $E$ is totally real.

Now let $p$ be a prime and let $\Sigma_{p}(E)$ denote the set of places of $E$ above $p$.
For each $w \in \Sigma_{p}(E)$ let $\log_{w}:U^{1}_{E_{w}} \rightarrow E_{w}$ be 
the restriction of the $p$-adic logarithm $\log_{p}$ and recall that $\ker(\log_{w})$ consists of the $p$-power roots of unity in $E_{w}$
(see \cite[Proposition 5.6]{MR1421575}).
The composite $\Z_{p}[G]$-module homomorphism
\[
\Z_{p} \otimes_{\Z} \log_{\infty}(\mathcal{O}_{E}^{\times}) 
\xrightarrow{{\exp}_{\infty}'} \Z_{p} \otimes_{\Z} \mathcal{O}_{E}^{\times} \xrightarrow{\lambda_{p}}
\prod_{w \in \Sigma_{p}(E)} U^{1}_{E_{w}} \xrightarrow{\prod_{w} \log_{w}} \prod_{w \in \Sigma_{p}(E)} E_{w} \cong \Q_{p} \otimes_{\Q} E
\]
factors through the inclusion $\Q_{p} \otimes_{\Q} E^{0} \subset \Q_{p} \otimes_{\Q} E$
and hence induces a homomorphism of $\C_{p}[G]$-modules 
\[
\mu_{p}: \C_{p} \otimes_{\Z} \log_{\infty} (\mathcal{O}_{E}^{\times}) \longrightarrow \C_{p} \otimes_{\Q} E^{0}.
\]
Observe that
\begin{equation}\label{eq:leo-equiv-mu-p-iso}
\Leo(E,p) \textrm{ holds} \Longleftrightarrow \Q_{p} \otimes_{\Z_{p}} \lambda_{p} \textrm{ is injective} \Longleftrightarrow \mu_{p} \textrm{ is an isomorphism},
\end{equation}
where the first equivalence is Theorem \ref{thm:leo-results} (ii) and the second equivalence follows from the definition of $\mu_{p}$, 
the injectivity of $\exp'_{\infty}$, the fact that $\ker(\log_{w})$ is torsion for each $w$, and that both the domain and codomain of $\mu_{p}$
are of $\C_{p}$-dimension $[E:\Q]-1$.

\begin{definition}
Let $j: \C \cong \C_{p}$ be a field isomorphism and let $\rho \in R_{\C_{p}}^{+}(G)$. 
We define the comparison period attached to $\rho$ and $j$ to be
\[ 
\Omega_{j}(\rho) := 
{\det}_{\C_{p}}(\mu_{p} \circ (\C_{p} \otimes_{\C,j} \mu_{\infty})^{-1})^{\rho} \in \C_{p}.
\] 
\end{definition}

We record some basic properties of $\Omega_{j}(-)$.

\begin{lemma}\label{lem:properties-of-omega}
Let $H,N$ be subgroups of $G$ with $N$ normal in $G$.
\begin{enumerate}
\item 
Let $\rho_{1}, \rho_{2} \in  R_{\C_{p}}^{+}(G)$. Then $\Omega_{j}(\rho_{1} + \rho_{2}) = \Omega_{j}(\rho_{1}) \Omega_{j}(\rho_{2})$.
\item
Let $\rho \in R_{\C_{p}}^{+}(H)$. Then $\Omega_{j}(\ind_{H}^{G} \rho) = \Omega_{j}(\rho)$.
\item
Let $\rho \in R_{\C_{p}}^{+}(G / N)$. Then $\Omega_{j}(\infl_{G / N}^{G} \rho) = \Omega_{j}(\rho)$.
\end{enumerate}
\end{lemma}

\begin{proof}
Each part follows from the corresponding part of Lemma \ref{lem:properties-of-rho-parts-functor}.
\end{proof}

\begin{remark}\label{rmk:vanishing-of-mu-p-independent-of-j}
Since $\mu_{\infty}$ is an isomorphism, for any two choices of field isomorphism $j,j': \C \cong \C_{p}$ we have that 
$\Omega_{j}(\rho)=0$ if and only if $\Omega_{j'}(\rho)=0$.
\end{remark}

\begin{remark}\label{rmk:rho-parts-of-Leo}
For any fixed choice of field isomorphism $j:\C \cong \C_{p}$ we have 
\begin{eqnarray*}
\Leo(E,p) \textrm{ holds} 
&\Longleftrightarrow& \mu_{p} \textrm{ is an isomorphism} \\
&\Longleftrightarrow& \Omega_{j}(\rho) \neq 0 \quad \forall \rho \in \Irr_{\C_{p}}(G)\\
&\Longleftrightarrow& \Omega_{j}(\rho) \neq 0 \quad \forall \rho \in R_{\C_{p}}^{+}(G),
\end{eqnarray*}
where the first equivalence is \eqref{eq:leo-equiv-mu-p-iso} and the last follows from Lemma \ref{lem:properties-of-omega} (i).
Thus the non-vanishing of $\Omega_{j}(\rho)$ can be thought of as the `$\rho$-part' of $\Leo(E,p)$.
Moreover, if $\Omega_{j}(\rho) \neq 0$ then we may set $\Omega_{j}(-\rho) := \Omega_{j}(\rho)^{-1}$
and so if we assume $\Leo(E,p)$ then Lemma \ref{lem:properties-of-omega} (i) shows that the definition of $\Omega_{j}(\rho)$ 
naturally extends to any virtual character $\rho \in R_{\C_{p}}(G)$.
\end{remark}

\subsection{$p$-adic Artin $L$-functions}\label{subsec:p-adic-Artin-L-functions}
Let $E/F$ be a finite Galois extension of totally real number fields and let $G=\Gal(E/F)$. Let $p$ be a prime
and let $\Sigma$ be a finite set of places of $F$ containing $\Sigma_{p}(F) \cup \Sigma_{\infty}(F)$.
For each character $\rho \in R_{\C_{p}}(G)$ the $\Sigma$-truncated $p$-adic Artin $L$-function 
attached to $\rho$ is the unique $p$-adic meromorphic
function $L_{p,\Sigma}(s, \rho) : \Z_{p} \rightarrow \C_{p}$ with the property that for each strictly negative 
integer $n$ and each field isomorphism $j: \C \cong \C_{p}$ we have
\[
L_{p,\Sigma}(n, \rho) = j \left( L_{\Sigma}(n, (\rho \otimes \omega^{n-1})^{j^{-1}}) \right),
\]
where $\omega: \Gal(\overline{F}/F) \rightarrow \Z_{p}^{\times}$ is the Teichm\"uller character.
(By a result of Siegel \cite{MR0285488} the right-hand side does not depend on the choice of $j$.)
These functions satisfy the same properties with respect to induction, inflation and addition of characters as complex Artin $L$-functions. 
In the case that $\rho$ is linear, $L_{p,\Sigma}(s, \rho)$ was constructed independently by 
Deligne and Ribet \cite{MR579702}, Barsky \cite{MR525346} and Cassou-Nogu\'es \cite{MR524276}.
Greenberg \cite{MR692344} then extended the construction to the general case using Brauer induction.
Note that when $\Sigma = \Sigma_{p}(F) \cup \Sigma_{\infty}(F)$ we have $L_{p,\Sigma}(s,\mathbbm{1}_{G}) = \zeta_{F,p}(s)$
(see \S \ref{subsec:p-adic-ACNF}).
If we assume $\Leo(E,p)$ and write $L_{p,\Sigma}^{*}(1,\rho)$ for the leading term at $s=1$ of $L_{p,\Sigma}(s,\rho)$,
then in analogy with \eqref{eq:Artin-L-order-of-vanishing} one can show that
\begin{equation}\label{eq:p-adic-L-fnc-leading-term}
L_{p,\Sigma}^{*}(1,\rho) = \lim_{s \rightarrow 1} (s-1)^{\langle \mathbbm{1}_{G}, \rho \rangle_{G}} \cdot L_{p,\Sigma}(s,\rho). 
\end{equation}

\subsection{Statement of the $p$-adic Stark conjecture at $s=1$}\label{subsec:formulation-p-adic-Stark}
This conjecture is discussed by Tate in \cite[Chapitre VI, \S 5]{MR782485} where it is attributed to Serre; in fact, it appears to 
be Tate's elaboration of a remark in \cite{MR0506177}.
Solomon \cite[\S 3.3]{MR1906480} noted some slight imprecisions in Tate's discussion and also gave an alternative formulation in the case that 
$G$ is abelian.
Burns and Venjakob clarified Tate's formulation in the general case 
in \cite[\S 5.2]{MR2290587} and gave a simplified presentation in \cite[\S 7.1]{MR2749572}.  
We shall use the following slight variant of the latter version as it will be the most useful for applications.

\begin{conj}[The $p$-adic Stark conjecture at $s=1$]\label{conj:p-adic-Stark-at-s=1}
Let $E/F$ be a finite Galois extension of totally real number fields and let $G=\Gal(E/F)$.
Let $p$ be a prime and let $\Sigma$ be a finite set of places of $F$ containing $\Sigma_{p}(F) \cup \Sigma_{\infty}(F)$.
Let $\rho \in R_{\C_{p}}^{+}(G)$. Then for every choice of field isomorphism $j: \C \cong \C_{p}$ we have 
\begin{equation}\label{eq:p-adic-Stark-at-s=1-eq}
L_{p, \Sigma}^{*}(1,\rho) = \Omega_{j}(\rho) \cdot j \left( L_{\Sigma}^{*}(1, \rho^{j^{-1}}) \right). 
\end{equation}
\end{conj}

\begin{remark}\label{rmk:indep-choice-of-sigma}
It is clear that $\Omega_{j}(\rho)$ does not depend on $\Sigma$. 
Thus, by considering Euler factors, it is straightforward to show that the
truth of Conjecture \ref{conj:p-adic-Stark-at-s=1} is independent of the choice of $\Sigma$.
\end{remark}

\begin{remark}\label{rmk:extend-p-adic-Stark-to-virtual-chars}
If $\Omega_{j}(\rho) \neq 0$ then by setting $\Omega_{j}(-\rho) := \Omega_{j}(\rho)^{-1}$ we see that 
the statement of Conjecture \ref{conj:p-adic-Stark-at-s=1} naturally extends to the virtual character $-\rho \in R_{\C_{p}}(G)$.
In particular, if we assume $\Leo(E,p)$ then Remark \ref{rmk:rho-parts-of-Leo} shows that the statement of Conjecture 
\ref{conj:p-adic-Stark-at-s=1} extends to all virtual characters $\rho \in R_{\C_{p}}(G)$.
\end{remark}

\begin{remark}\label{rmk:invariance-of-p-adic-Stark}
Since both complex and $p$-adic Artin $L$-functions satisfy properties analogous to those 
of $\Omega_{j}(-)$ given in Lemma \ref{lem:properties-of-omega}, 
the truth of Conjecture \ref{conj:p-adic-Stark-at-s=1} is invariant under induction and inflation;
moreover, if it holds for $\rho_{1}, \rho_{2} \in R_{\C_{p}}(G)$ then it holds for $\rho_{1}+\rho_{2}$.
\end{remark}

\begin{remark}\label{rmk:p-adic-Stark-closed-under-Z-linear-combs}
Since leading terms are non-zero by definition, the equality \eqref{eq:p-adic-Stark-at-s=1-eq} implies that $\Omega_{j}(\rho) \neq 0$.
Thus if \eqref{eq:p-adic-Stark-at-s=1-eq} holds then taking reciprocals shows that it holds with $\rho$ replaced by $-\rho$.  
By combining this observation with Remarks \ref{rmk:extend-p-adic-Stark-to-virtual-chars} and \ref{rmk:invariance-of-p-adic-Stark} 
we see that the subset of $R_{\C_{p}}(G)$ for which Conjecture \ref{conj:p-adic-Stark-at-s=1} holds is closed under $\Z$-linear 
combinations (this property will be crucial for several proofs later on).
In particular, if Conjecture \ref{conj:p-adic-Stark-at-s=1} holds for all $\rho \in \Irr_{\C_{p}}(G)$ then it holds for all $\rho \in R_{\C_{p}}(G)$.
\end{remark}

\begin{remark}\label{rmk:p-adic-Stark-for-all-chars-implies-Leopoldt}
Remark \ref{rmk:rho-parts-of-Leo} and the first sentence of Remark \ref{rmk:p-adic-Stark-closed-under-Z-linear-combs}
show that if Conjecture \ref{conj:p-adic-Stark-at-s=1} holds for all $\rho \in \Irr_{\C_{p}}(G)$ then $\Leo(E,p)$ holds.  
\end{remark}

\begin{remark}\label{rmk:p-adic-abelian-subextens-reduction}
Remarks \ref{rmk:invariance-of-p-adic-Stark} and \ref{rmk:p-adic-Stark-closed-under-Z-linear-combs}
together with Brauer's theorem on induced characters (see \cite[\S 15B]{MR632548}, for example)
show that the proof of Conjecture \ref{conj:p-adic-Stark-at-s=1} for $E/F$ and a fixed choice of $p$ reduces 
to certain cyclic sub-extensions of $E/F$.
\end{remark}

\subsection{The relation to Stark's conjecture at $s=1$}

\begin{theorem}\label{thm:p-stark-implies-stark}
Let $E/F$ be a finite Galois extension of totally real number fields and let $G=\Gal(E/F)$.
Let $p$ be a prime and let $\Sigma$ be a finite set of places of $F$ containing $\Sigma_{p}(F) \cup \Sigma_{\infty}(F)$.
Let $\rho \in R_{\C_{p}}^{+}(G)$.
If $\Omega_{j}(\rho) \neq 0$ for some (and hence every) choice of field isomorphism $j:\C \cong \C_{p}$ (in particular, this is the case if $\Leo(E,p)$ holds) 
then the following statements are equivalent.
\begin{enumerate}
\item $\Omega_{j}(\rho) \cdot j ( L_{\Sigma}^{*}(1, \rho^{j^{-1}}) )$ is independent of the choice of $j:\C \cong \C_{p}$.
\item Stark's conjecture at $s=1$ holds for $\rho^{j^{-1}} \in R_{\C}^{+}(G)$ and some (and hence every) choice of $j:\C \cong \C_{p}$.
\end{enumerate}
\end{theorem}

\begin{remark}
That (ii) implies (i) in Theorem \ref{thm:p-stark-implies-stark} was already shown by Serre (see \cite[Chapitre VI, Th\'eor\`eme 5.2]{MR782485}); it is clear that this does not require
the hypothesis that $\Omega_{j}(\rho) \neq 0$ for some (and hence every) choice of $j$.
\end{remark}

\begin{proof}[Proof of Theorem \ref{thm:p-stark-implies-stark}]
The first and second occurrence of `and hence every' in statement of the theorem follow from Remark 
\ref{rmk:stark-1-independent-of-choices} and Remark \ref{rmk:vanishing-of-mu-p-independent-of-j} (iii), respectively.

Let $j,j': \C \cong \C_{p}$ be field isomorphisms and let $\chi := \rho^{j^{-1}}$.
Then $j = j' \circ \sigma$ for some $\sigma \in \Aut(\C)$ and so $\rho^{j'^{-1}} = \chi^{\sigma}$.
For every $\Q[G]$-isomorphism
$g: E^{0} \xrightarrow{\, \sim \,} \Q \otimes_{\Z} \log_{\infty}(\mathcal{O}_{E}^{\times})$ we have
\[
j \left( R_{1}(\chi, g) \right) = j \left({\det}_{\C} (\mu_{\infty} \circ (\C \otimes_{\Q} g))^{\chi}\right) = 
{\det}_{\C_{p}} ((\C_{p} \otimes _{\C,j}\mu_{\infty}) \circ (\C_{p} \otimes_{\Q} g))^{\rho},
\]
and thus
\[
\Omega_{j}(\rho) \cdot j \left( R_{1}(\chi, g) \right) = {\det}_{\C_{p}} (\mu_{p} \circ (\C_{p} \otimes_{\Q} g))^{\rho},
\]
which does not depend on $j$. 
In particular, $\Omega_{j}(\rho) \cdot j \left( R_{1}(\chi, g) \right) = \Omega_{j'}(\rho) \cdot j' \left( R_{1}(\chi^{\sigma}, g) \right)$.
Hence
\[
\frac{\Omega_{j}(\rho) \cdot j ( L_{\Sigma}^{*}(1, \rho^{j^{-1}}) )}{\Omega_{j'}(\rho) \cdot  j' ( L_{\Sigma}^{*}(1, \rho^{(j')^{-1}}) )}
= \frac{j'(R_{1}(\chi^{\sigma},g)) \cdot j(L_{\Sigma}^{\ast}(1, \chi))}{j(R_{1}(\chi, g)) \cdot j'(L_{\Sigma}^{\ast}(1, \chi^{\sigma}))}
= j' \left( \frac{\sigma(L_{\Sigma}^{\ast}(1, \chi)) \cdot R_{1}(\chi^{\sigma},g)}
{\sigma(R_{1}(\chi, g)) \cdot L_{\Sigma}^{\ast}(1, \chi^{\sigma})} \right),
\]
which is equal to $1$ if and only if Stark's conjecture at $s=1$ (Conjecture \ref{conj:Stark-at-1}) holds for the character $\chi$.
\end{proof}

The following result is perhaps counter-intuitive because its hypotheses depend on a fixed prime $p$, yet its conclusion does not.
It will be crucial for the proof of the prime-by-prime descent result of Theorem \ref{thm:descent-result}.

\begin{corollary}\label{cor:p-stark-implies-stark}
Let $E/F$ be a finite Galois extension of totally real number fields and let $G=\Gal(E/F)$.
Fix a prime $p$. If the $p$-adic Stark conjecture at $s=1$ holds for all $\rho \in R_{\C_{p}}^{+}(G)$ 
then Stark's conjecture at $s=1$ holds for all $\chi \in R_{\C}^{+}(G)$.  
\end{corollary}

\begin{proof}
Let $\rho \in R_{\C_{p}}^{+}(G)$.
The $p$-adic Stark conjecture at $s=1$ for $\rho$
implies that $\Omega_{j}(\rho) \neq 0$ for every choice of $j$ (see Remark \ref{rmk:p-adic-Stark-closed-under-Z-linear-combs})
and that condition (i) of Theorem \ref{thm:p-stark-implies-stark} holds.
Thus Stark's conjecture at $s=1$ holds for $\rho^{j^{-1}} \in R_{\C}^{+}(G)$ for every choice of $j$.
The desired result now follows from the observation that every $\chi \in R_{\C}^{+}(G)$ can be written in the form 
$\chi = \rho^{j^{-1}}$ for some $\rho \in R_{\C_{p}}^{+}(G)$ and some $j: \C \cong \C_{p}$.
\end{proof}

\subsection{An alternative description of the comparison period}\label{subsec:alt-description-of-omega}
Let $E/F$ be a finite Galois extension of totally real number fields and let $G=\Gal(E/F)$.
Let $p$ be a prime. 
For any $j:\C \cong \C_{p}$ and any $\rho \in R_{\C_{p}}^{+}(G)$, we shall define a period 
$\Psi_{j}(\rho)$ and show that $\Psi_{j}(\rho)=\Omega_{j}(\rho)$. 
This alternative description of $\Omega_{j}(\rho)$ will be used in \S \ref{sec:rational-valued-characters} on rational-valued characters and in 
\S \ref{sec:p-adic-Stark-abs-abelian} on absolutely abelian characters. However, we emphasise that the results of this section are valid for all characters.

If we view $\C \otimes_{\Q} E^{0}$ as a submodule of $\prod_{\Hom(E,\C)} \C$ via the isomorphism 
$\C \otimes_{\Q} E \cong \prod_{\Hom(E,\C)} \C$, then we can say that the usual Dirichlet map
\begin{eqnarray*}
	\varphi_{\infty}: \C \otimes_{\Z} \mathcal{O}_{E}^{\times} & \stackrel{\sim}{\longrightarrow} 
		& \C \otimes_{\Q} E^{0} \\
	z \otimes \epsilon & \mapsto &  \left(z\log |\sigma(\epsilon)|\right)_{\sigma \in \Hom(E,\C)}, 
\end{eqnarray*}
is a canonical isomorphism of $\C[G]$-modules.
We likewise define a $p$-adic Dirichlet map
\begin{eqnarray*}
	\varphi_{p}: \C_{p} \otimes_{\Z} \mathcal{O}_{E}^{\times} & \longrightarrow
		& \C_{p} \otimes_{\Q} E^{0} \\
	z \otimes \epsilon & \mapsto & \left(z\log_{p} (\tau(\epsilon))\right)_{\tau \in \Hom(E,\C_{p})},
\end{eqnarray*}
which is a homomorphism of $\C_{p}[G]$-modules. 
By Theorem \ref{thm:leo-results} (v) and the definitions of $\varphi_{p}$ and $R_{E,p}$ we have
\begin{equation*}\label{eq:leo-equiv-varphi-p-iso}
\Leo(E,p) \textrm{ holds} \Longleftrightarrow R_{E,p} \neq 0 \Longleftrightarrow \varphi_{p} \textrm{ is an isomorphism}. 
\end{equation*}

\begin{definition}
Let $j: \C \cong \C_{p}$ be a field isomorphism and let $\rho \in R_{\C_{p}}^{+}(G)$. 
We define
\[ 
\Psi_{j}(\rho) := 
{\det}_{\C_{p}}(\varphi_{p} \circ (\C_{p} \otimes_{\C,j} \varphi_{\infty})^{-1})^{\rho} \in \C_{p}.
\] 
\end{definition}

\begin{lemma}\label{lem:psi-equal-to-phi}
Let $j: \C \cong \C_{p}$ be a field isomorphism and let $\rho \in R_{\C_{p}}^{+}(G)$. Then
\[
\varphi_{p} \circ (\C_{p} \otimes_{\C, j} \varphi_{\infty})^{-1} = \mu_{p} \circ (\C_{p} \otimes_{\C, j} \mu_{\infty})^{-1},
\]
and so in particular $\Psi_{j}(\rho)=\Omega_{j}(\rho)$. 
\end{lemma}

\begin{proof}
Let $\epsilon \in \mathcal{O}_{F}^{\times}$. It suffices to show that
\[
\varphi_{p}(1 \otimes \epsilon) = \mu_{p} \circ (\C_{p} \otimes_{\C, j} \mu_{\infty})^{-1}  \circ (\C_{p} \otimes_{\C, j} \varphi_{\infty})(1 \otimes \epsilon).
\]
In fact, since the subgroup of $\mathcal{O}_{F}^{\times}$ of totally positive units is of finite index, 
we can and do assume without loss of generality that $\epsilon$ is totally positive. Thus we have
\[
(\C_{p} \otimes_{\C, j} \varphi_{\infty})(1 \otimes \epsilon) 
= (j (\log|\sigma(\epsilon)|))_{\sigma \in \Hom(E,\C)}
= (j (\log \sigma(\epsilon)))_{\sigma \in \Hom(E,\C)} 
\in \C_{p} \otimes_{\Q} E^{0},
\]
where we have used the identification
\[
\C_{p} \otimes_{\Q} E^{0} \subset \C_{p} \otimes_{\Q} E \cong \C_{p} \otimes_{\C,j} {\textstyle \prod_{\Hom(E,\C)}} \C \cong {\textstyle \prod_{\Hom(E,\C_{p})}} \C_{p}.
\]
However, we also have 
\[
\C_{p} \otimes_{\Z} \log_{\infty}(\mathcal{O}_{E}^{\times}) \subset \C_{p} \otimes_{\Q} E \cong {\textstyle \prod_{\Hom(E,\C_{p})}} \C_{p}
\]
and the canonical isomorphism 
$(\C_{p} \otimes_{\C,j} \mu_{\infty})^{-1} : \C_{p} \otimes_{\Q} E^{0} \cong \C_{p} \otimes_{\Z} \log_{\infty}(\mathcal{O}_{E}^{\times})$
can be considered as an equality inside $\C_{p} \otimes_{\Q} E \cong {\textstyle \prod_{\Hom(E,\C_{p})}} \C_{p}$.
Therefore
\begin{align}\label{eq:first-epsilon-calc}
(\C_{p} \otimes_{\C,j} \mu_{\infty})^{-1} \circ (\C_{p} \otimes_{\C, j} \varphi_{\infty})(1 \otimes \epsilon) &=
(j (\log \sigma(\epsilon)))_{\sigma \in \Hom(E,\C)} \\
&\in \C_{p} \otimes_{\Z} \log_{\infty}(\mathcal{O}_{E}^{\times})  \subset {\textstyle \prod_{\Hom(E,\C_{p})}} \C_{p}. \nonumber
\end{align}

Since $\epsilon$ is totally positive we have
\begin{eqnarray*}
\textstyle{\prod_{\sigma \in \Hom(E,\R)}} \R \cong E_{\infty} &\xrightarrow{\, \exp_{\infty} \, }& E_{\infty}^{\times} \subset E_{\infty} 
\cong  \textstyle{\prod_{\sigma \in \Hom(E,\R)}} \R  \\
(\log \sigma(\epsilon))_{\sigma} & \mapsto & (\sigma(\epsilon))_{\sigma}.
\end{eqnarray*}
Moreover, we also have
\begin{eqnarray*}
\mathcal{O}_{E}^{\times} & \hookrightarrow & E_{\infty}^{\times} \subset E_{\infty} 
\cong  \textstyle{\prod_{\sigma \in \Hom(E,\R)}} \R  \\
\epsilon & \mapsto & (\sigma(\epsilon))_{\sigma}. \nonumber
\end{eqnarray*}
Hence 
\begin{equation}\label{eq:second-epsilon-calc}
(\id_{\C_{p}} \otimes_{\Z} \exp_{\infty}')((j (\log \sigma (\epsilon)))_{\sigma \in \Hom(E,\C)}) 
= 1 \otimes \epsilon \in \C_{p} \otimes_{\Z} \mathcal{O}_{E}^{\times} 
\end{equation}
where $\exp_{\infty}':\log_{\infty}(\mathcal{O}_{E}^{\times}) \rightarrow \mathcal{O}_{E}^{\times}$ 
was defined in \S \ref{subsec:comparison-period}.
Therefore combining \eqref{eq:first-epsilon-calc} and \eqref{eq:second-epsilon-calc} gives
\[
(\id_{\C_{p}} \otimes_{\Z_{p}} \exp_{\infty}') \circ (\C_{p} \otimes_{\C,j} \mu_{\infty})^{-1} \circ (\C_{p} \otimes_{\C, j} \varphi_{\infty})
= 
\id_{\C_{p} \otimes_{\Z} \mathcal{O}_{E}^{\times}}.
\]

Now
\[
\lambda_{p}(1 \otimes \epsilon) = (\sigma_{w}(\epsilon))_{w \in \Sigma_{p}(E)} \in  \prod_{w \in \Sigma_{p}(E)} U^{1}_{E_{w}}
\]
where $\sigma_{w}: E \rightarrow E_{w}$ is the embedding of $E$ into its completion $E_{w}$. Moreover, we have
\[
\arraycolsep=2pt
\begin{array}{ccccc}
\displaystyle
\prod_{w \in \Sigma_{p}(E)} U^{1}_{E_{w}} & 
\xrightarrow{\prod_{w} \log_{w}} & 
\displaystyle
\prod_{w \in \Sigma_{p}(E)} E_{w} & \cong \Q_{p} \otimes_{\Q} E \subset \C_{p} \otimes_{\Q} E   \cong & 
\displaystyle
\prod_{w \in \Sigma_{p}(E)} \prod_{\tau_{w} \in \Hom(E_{w},\C_{p})} \C_{p} \\
(\sigma_{w}(\epsilon))_{w} &  \xmapsto{\phantom{\prod_{w} \log_{w}}} & 
(\log_{w} (\sigma_{w}(\epsilon)))_{w} 
& \xmapsto{\qquad \qquad \qquad \qquad \qquad} & 
(\tau_{w}( \log_{w} (\sigma_{w}(\epsilon))))_{w,\tau_{w}} .
\end{array}
\]
Finally, we have 
\[
\arraycolsep=2pt
\begin{array}{ccl}
\displaystyle
\prod_{w \in \Sigma_{p}(E)} \prod_{\tau_{w} \in \Hom(E_{w},\C_{p})} \C_{p} &
\cong & 
\displaystyle
\prod_{\tau \in \Hom(E,\C_{p})} \C_{p} \quad \textrm{ via } \quad  (w,\tau_{w}) \mapsto \tau = \tau_{w} \circ \sigma_{w} \\
(\tau_{w} (\log_{w} (\sigma_{w}(\epsilon))))_{w,\tau_{w}}
& \mapsto &
(\log_{p}(\tau(\epsilon)))_{\tau}  = \varphi_{p}(1 \otimes \epsilon) ,
\end{array}
\]
which gives the desired result.
\end{proof}

\section{Rational-valued characters}\label{sec:rational-valued-characters}

\subsection{The trivial character}
Let $E/F$ be a finite Galois extension of totally real number fields and let $G=\Gal(E/F)$.
Let $p$ be a prime. 

\begin{prop}[{\cite[Remark p.\ 138]{MR782485}}] \label{prop:p-Stark-for-trivial-character}
$\Leo(F,p)$ holds if and only if the $p$-adic Stark conjecture at $s=1$ holds for the trivial character $\mathbbm{1}_{G}$. 
\end{prop}

\begin{proof}
Let $\Sigma = \Sigma_{p}(F) \cup \Sigma_{\infty}(F)$. Then by Remark \ref{rmk:indep-choice-of-sigma} and Lemma \ref{lem:psi-equal-to-phi} the 
$p$-adic Stark conjecture at $s=1$ for $\mathbbm{1}_{G}$ is equivalent to the assertion that the equality
\[
L^{*}_{p,\Sigma}(1,\mathbbm{1}_{G}) = \Psi_{j}(\mathbbm{1}_{G}) \cdot j \left( L^{*}_{\Sigma}(1,\mathbbm{1}_{G}) \right)
\]
holds for every field isomorphism $j: \C \cong \C_{p}$. Moreover, $\Psi_{j}(\mathbbm{1}_{G}) = R_{F,p} \cdot j(R_{F,\infty})^{-1}$.
Hence, by using Remark \ref{rmk:sign-of-quotient} and Corollary \ref{cor:simple-pole-iff-leopoldt},
the desired result follows by comparing the $p$-adic analytic class number formula at $s=1$ 
(Theorem \ref{thm:p-adic-class-number-formula}) with the usual analytic class number formula at $s=1$.
\end{proof}

\begin{corollary}[{\cite[Remark 5.4]{MR2290587}}]\label{cor:p-Stark-for-induced-chars}
Let $H$ be a subgroup of $G$ and let $\rho = \ind_{H}^{G} \mathbbm{1}_{H}$.
Then $\Leo(E^{H},p)$ holds if and only if  the $p$-adic Stark conjecture at $s=1$ holds for $\rho$. 
\end{corollary}

\begin{proof}
This is just the combination of Proposition \ref{prop:p-Stark-for-trivial-character} and the fact that the truth of the $p$-adic Stark conjecture at $s=1$
is invariant under induction (see Remark \ref{rmk:invariance-of-p-adic-Stark}).
\end{proof}

\subsection{Permutation characters and rational-valued characters}
Let $G$ be a finite group and let $\rho \in  \Char_{\Q}(G)$.
Then by Artin's induction theorem (see \cite[(15.4)]{MR632548}, for instance)
there exists a natural number $n_{\rho}$ dividing $|G|$ such that 
\begin{equation}\label{eq:virtual-expression}
n_{\rho} \cdot \rho = \sum n_{H} \cdot \ind^{G}_{H} \mathbbm{1}_{H}
\end{equation}
where the sum runs over all subgroups $H$ of $G$ and each $n_{H}$ is an integer. 
By definition $\rho \in \Perm(G)$ if and only if one can take $n_{\rho}=1$.

The following result is analogous to but different from \cite[Corollary 5.7]{MR2290587}, which relates
leading terms of $p$-adic Artin $L$-functions at $s=1$ to those of `global Zeta isomorphisms'. 

\begin{prop}\label{prop:leo-implies-p-adic-Stark-for-virtual-perm-chars}
Let $E/F$ be a finite Galois extension of totally real number fields and let $G=\Gal(E/F)$.
Let $p$ be a prime and let $\Sigma$ be a finite set of places of $F$ containing $\Sigma_{p}(F) \cup \Sigma_{\infty}(F)$.
Suppose that $\Leo(E,p)$ holds. Let $\rho \in \Char_{\Q}(G) \subset R_{\C_{p}}(G)$ and suppose that 
the expression \eqref{eq:virtual-expression} holds for $\rho$.
Then for every field isomorphism $j: \C \cong \C_{p}$ we have 
\[
\left( L_{p, \Sigma}^{*}(1,\rho) \right)^{n_{\rho}} = \left(\Omega_{j}(\rho) \cdot j\left( L_{\Sigma}^{*}(1, \rho^{j^{-1}})\right)\right)^{n_{\rho}}. 
\]
In particular, if $\rho \in \Perm(G)$, then the $p$-adic Stark conjecture at $s=1$ holds for $\rho$.
\end{prop}

\begin{proof}
By Theorem \ref{thm:leo-results} (iv), $\Leo(E^{H},p)$ holds for every subgroup $H \leq G$.
Thus by Corollary \ref{cor:p-Stark-for-induced-chars}, the $p$-adic Stark conjecture at $s=1$ holds for every $\ind^{G}_{H} \mathbbm{1}_{H}$.
Hence Remark \ref{rmk:p-adic-Stark-closed-under-Z-linear-combs} and expression \eqref{eq:virtual-expression} show that it holds for $n_{\rho} \cdot \rho$.
The desired result now follows from the  properties of $\Omega_{j}(-)$ and of both complex and $p$-adic Artin $L$-functions with respect to addition of characters.
\end{proof}

\begin{corollary}\label{cor:leo-implies-p-adic-Stark-for-virtual-perm-chars}
Let $E/F$ be a finite Galois extension of totally real number fields such that $G:=\Gal(E/F)$ satisfies $\Perm(G)= R_{\C}(G)$.
Let $p$ be a prime and suppose that  $\Leo(E,p)$ holds.
Then the $p$-adic Stark conjecture at $s=1$ holds for all $\rho \in R_{\C_{p}}(G)$.
\end{corollary}

\begin{remark}
The condition $\Perm(G)= R_{\C}(G)$ is discussed in Remark \ref{rmk:all-chars-perm-chars}. 
\end{remark}

\section{Absolutely abelian characters}\label{sec:p-adic-Stark-abs-abelian}

We now prove the $p$-adic Stark conjecture at $s=1$ for absolutely abelian characters by building on 
work of Ritter and Weiss \cite[\S 10]{MR1423032} and using standard results on Dirichlet $L$-functions and Kubota-Leopoldt $p$-adic $L$-functions 
(see \cite{MR1421575}, for example).

\begin{remark}\label{rmk:solomon-abs-abelian}
In \cite[\S 3.5]{MR1906480}, Solomon proved a refinement \cite[Conjecture 3.6]{MR1906480} of his version of the $p$-adic Stark conjecture at $s=1$
\cite[Conjecture 3.3]{MR1906480} in the absolutely abelian case, but under the additional hypothesis that $p$ does not divide the conductor of the corresponding Dirichlet character. 
\end{remark}

\begin{theorem}\label{thm:p-adic-Stark-for-absolutely-abelian-characters}
Let $E/F$ be a finite Galois extension of totally real number fields and let $G=\Gal(E/F)$. Let $p$ be a prime.
Suppose that $\rho \in R_{\C_{p}}(G)$ is an absolutely abelian character, i.e.,
there exists a normal subgroup $N$ of $G$ such that $\rho$ factors through $G/N \cong \Gal(E^{N}/F)$ and $E^{N}/\Q$ is abelian. 
Then the $p$-adic Stark conjecture at $s=1$ holds for $\rho$.
\end{theorem}

\begin{proof}
We first observe that we can make a number of simplifying assumptions.
Recall from Remark \ref{rmk:invariance-of-p-adic-Stark} that the truth of the $p$-adic Stark conjecture at $s=1$
is invariant under induction and inflation. By invariance under induction we may assume that $F=\Q$.
Moreover, by Remark \ref{rmk:p-adic-Stark-closed-under-Z-linear-combs} we may also assume that $\rho$ is irreducible. 
By Proposition \ref{prop:p-Stark-for-trivial-character} and the fact that $\Leo(\Q,p)$ holds, we need only consider the case in which $\rho$ is non-trivial.
Finally, by the Kronecker-Weber theorem and invariance under inflation we may further assume that $\rho$ is a Dirichlet character 
and that $E = \Q(\zeta_{n})^{+}$, the maximal totally real subfield of $\Q(\zeta_{n})$, where $n$ is the conductor of $\rho$
and $\zeta_{n}$ denotes a fixed primitive $n$-th root of unity. 
Note that $\Leo(E,p)$ holds by Theorem \ref{thm:leo-results} (iii).

Note that $n \not\equiv 2 \bmod{4}$ and let $n = \prod_{i=1}^{s} p_{i}^{e_i}$ be its prime factorisation.
Following \cite[Theorem 8.3]{MR1421575} and \cite[\S 10]{MR1423032},
let $I$ run through all proper subsets of $\{1, \dots, s\}$ and set $n_{I} := \prod_{i \in I} p_{i}^{e_i}$. 
Let
\[
	\xi_{n} := \prod_{I} (1 - \zeta_{n}^{n_{I}}) (1 - \zeta_{n}^{-n_{I}}) =  \prod_{I}  \left( 2 - (\zeta_{n}^{n_{I}} + \zeta_{n}^{-n_{I}})\right) \in E^{\times}.
\]
Observe that $\xi_{n}$ is totally positive. Moreover, for every $\sigma \in G$, 
each $(1 - \sigma(\zeta_{n})^{n_{I}})/(1 - \zeta_{n}^{n_{I}})$ is a cyclotomic unit
and thus $\xi_{n}^{\sigma - 1} := \sigma(\xi_{n})/\xi_{n}$ is a totally positive element of $\mathcal{O}_{E}^{\times}$.

Let $S_{\infty}$ denote the set of archimedean places of $E$. The map $\zeta_{n} \mapsto \exp(2 \pi i / n)$
embeds $\Q(\zeta_{n})$ into $\C$, and its restriction to $E$ gives a distinguished archimedean place $\infty \in S_{\infty}$.
Then $\Z S_{\infty}$ is a free $\Z[G]$-module of rank $1$ with basis $\infty$. 
Let $\Delta S_{\infty}$ denote the kernel of the augmentation map $\Z S_{\infty} \rightarrow \Z$
which sends each place in $S_{\infty}$ to $1$ and note that $\{ \sigma - 1 \mid \sigma \in G, \, \sigma \neq 1\}$ is a $\Z$-basis for $\Delta S_{\infty}$.
Thus, following \cite[\S 10]{MR1423032}, 
the $\Z[G]$-linear map $\Z S_{\infty} \rightarrow E^{\times}$ which sends $\infty$ to $\xi_{n}$ induces an embedding
\begin{equation*}
	\phi: \Delta S_{\infty} \hookrightarrow \mathcal{O}_{E}^{\times}.
\end{equation*}
Let $e_{\rho} := |G|^{-1} \sum_{\sigma \in G} \rho(\sigma^{-1}) \sigma \in \C_{p}[G]$
be the primitive idempotent corresponding to $\rho$. 
As $\rho$ is a non-trivial linear character, we have
\begin{eqnarray*}
(\C_{p} \otimes_{\Z} \Delta S_{\infty})^{\rho} 
&=& \Hom_{\C_{p}[G]}(e_{\rho}\C_{p}[G], \C_{p} \otimes_{\Z} \Delta S_{\infty}) \\
&=& e_{\rho} \cdot (\C_{p} \otimes_{\Z} \Delta S_{\infty}) = e_{\rho} \cdot \C_{p} S_{\infty} = e_{\rho} \C_{p}[G] \cdot \infty. 
\end{eqnarray*}
Hence $e_{\rho} \infty$ is a $\C_{p}$-basis of $(\C_{p} \otimes_{\Z} \Delta S_{\infty})^{\rho}$.
Let $j: \C \cong \C_{p}$ be a field isomorphism and recall the definitions of $\varphi_{\infty}$ and $\varphi_{p}$ 
from \S \ref{subsec:alt-description-of-omega}. We denote the composite isomorphism of $\C_{p}[G]$-modules
\[
	\C_{p} \otimes_{\Z} \Delta S_{\infty} \xrightarrow{1 \otimes \phi} \C_{p} \otimes_{\Z} \mathcal{O}_{E}^{\times}
	\xrightarrow{1 \otimes_{\C, j} \varphi_{\infty}} \C_{p} \otimes_{\Q} E^{0} 
	\cong \C_{p} \otimes_{\Z} \Delta S_{\infty}
\]
again by $\varphi_{\infty} \phi$ and
compute the image of $e_{\rho} \infty$ under this map (steps justified below):
\begin{eqnarray}
	\varphi_{\infty} \phi(e_{\rho} \infty)
	& = & |G|^{-1} \sum_{\sigma \in G} 
			\rho(\sigma^{-1}) (1 \otimes_{\C,j} \varphi_{\infty})(\xi_{n}^{\sigma-1}) \nonumber \\
	& = & |G|^{-1} \sum_{\sigma \in G} \rho(\sigma^{-1}) 
			\sum_{\sigma' \in G} j\left(\log \left((\sigma')^{-1}(\xi_{n}^{\sigma-1})\right)\right) \cdot \sigma' \infty  \nonumber \\
	& = & |G|^{-1} \sum_{\sigma \in G} \sum_{\sigma' \in G} \rho\left((\sigma \sigma')^{-1}\right) 
			 j\left(\log \left((\sigma')^{-1}(\xi_{n}^{\sigma \sigma'-1})\right)\right) \cdot \sigma' \infty  \nonumber \\
	& = & |G|^{-1} \sum_{\sigma \in G} \sum_{\sigma' \in G} \rho(\sigma^{-1}) \rho((\sigma')^{-1}) 
			 j\left(\log \xi_{n}^{\sigma}\right) \cdot \sigma' \infty \label{eqn:phi-infinity-expression} \\
	&  & - |G|^{-1} \sum_{\sigma \in G} \sum_{\sigma' \in G} \rho(\sigma^{-1}) \rho((\sigma')^{-1}) 
			 j\left(\log \xi_{n}^{(\sigma')^{-1}}\right) \cdot \sigma' \infty  \nonumber \\
	& = & \left(\sum_{\sigma \in G} \rho(\sigma^{-1}) j\left(\log \xi_{n}^{\sigma}\right)\right) e_{\rho} \infty.  \nonumber
\end{eqnarray}
Here, the first equality holds because $\sum_{\sigma \in G} \rho(\sigma^{-1})$  vanishes
and we thus have an equality $e_{\rho} = |G|^{-1} \sum_{\sigma \in G} \rho(\sigma^{-1})(\sigma - 1)$;
the second equality holds by the definition of $\varphi_{\infty}$ and the fact that $\xi_{n}^{\sigma-1}$ is totally positive; the third and fourth equalities
are clear; and the last equality again follows from the vanishing of $\sum_{\sigma \in G} \rho(\sigma^{-1})$.
A similar computation shows that
\begin{equation}\label{eqn:phi-p-expression}
	\varphi_{p} \phi(e_{\rho} \infty) = 
		\left(\sum_{\sigma \in G} \rho(\sigma^{-1}) \log_{p} \left( j(\xi_{n}^{\sigma})\right)\right) e_{\rho} \infty. 
\end{equation}
Using equations \eqref{eqn:phi-infinity-expression} and \eqref{eqn:phi-p-expression} and Lemma \ref{lem:psi-equal-to-phi} we obtain
\begin{equation} \label{eqn:Omega-expression} 
	\Omega_{j}(\rho) = \Psi_{j}(\rho) = \frac{\sum_{\sigma \in G} \rho(\sigma^{-1}) \log_{p} \left( j(\xi_{n}^{\sigma})\right)}		
					{\sum_{\sigma \in G} \rho(\sigma^{-1}) j\left(\log \xi_{n}^{\sigma}\right)}.
\end{equation}

We identify $\Gal(\Q(\zeta_{n}) / \Q)$ with $(\Z / n\Z)^{\times}$
in the usual way and view $\rho$ as an even $\C_{p}$-valued Dirichlet character $(\Z / n\Z)^{\times} \rightarrow \C_{p}^{\times}$. 
Let $L_{p}(s,\rho)$ be the (non-truncated) Kubota-Leopoldt $p$-adic $L$-function attached to $\rho$.
Let $\chi := \rho^{j^{-1}} : (\Z / n\Z)^{\times} \rightarrow \C^{\times}$ 
be the classical even Dirichlet character corresponding to $\rho$ via $j$, let $\tau(\chi)$ be its Gauss sum and 
let $L(s,\chi)$ be the (non-truncated) Dirichlet $L$-function attached to $\chi$.
Let $\check \rho$ and $\check \chi$ be the contragredient characters of $\rho$ and $\chi$, respectively.
Finally, let $\Sigma$ be the set containing $p$ and the unique infinite place of $\Q$.

We now compute the denominator of the right-hand side in \eqref{eqn:Omega-expression} (justifications below):
\begin{eqnarray}
	\sum_{\sigma \in G} \rho(\sigma^{-1}) j\left(\log \xi_{n}^{\sigma}\right)
		& = & \sum_{\sigma \in G} \rho(\sigma^{-1}) j\left(\sum_{I} \log |(1 - \zeta_{n}^{n_{I}})^{2 \sigma}| \right) \nonumber \\
	& = & \frac{1}{2} \sum_{a \in (\Z / n\Z)^{\times}} \check \rho(a) j\left(\sum_{I} \log |(1 - \zeta_{n}^{a n_{I}})^{2}| \right) \nonumber\\
	& = & j\left( \sum_{a \in (\Z / n\Z)^{\times}} \check \chi(a) \sum_{I} \log |1 - \zeta_{n}^{a n_{I}}| \right) \nonumber\\
	& = & j\left( \sum_{a \in (\Z / n\Z)^{\times}} \check \chi(a) \log |1 - \zeta_{n}^{a}| \right) \label{eqn:num-omega-expression} \\
	& = & - j\left( \frac{n}{\tau(\chi)} L(1, \chi)\right) \nonumber \\
	& = & - j\left(\frac{n}{\tau(\chi)} L_{\Sigma}(1, \chi)\right) \left(1 - \frac{\rho(p)}{p}\right)^{-1}. \nonumber
\end{eqnarray}
Here, the first equality holds because $\xi_{n} = \prod_{I}(-\zeta_{n}^{-n_{I}})(1 - \zeta_{n}^{n_{I}})^{2}$, the fourth equality follows from
\cite[Lemma 8.4]{MR1421575} and the fifth from \cite[Theorem 4.9]{MR1421575}. A similar computation using
\cite[Theorem 5.18]{MR1421575} shows that we have
\begin{equation}\label{eq:denom-omega-expression}
	\sum_{\sigma \in G} \rho(\sigma^{-1}) \log_{p} \left( j(\xi_{n}^{\sigma})\right) 
	= - j\left(\frac{n}{\tau(\chi)}\right) L_{p}(1, \rho) \left(1 - \frac{\rho(p)}{p}\right)^{-1}. 
\end{equation}
Moreover, $L(1,\chi) \neq 0$ by \cite[Corollary 4.4]{MR1421575} and $L_{p}(1, \rho) \neq 0$ by
 \cite[Corollary 5.30]{MR1421575}.
Hence $L_{\Sigma}(1,\chi)=L_{\Sigma}^{*}(1,\chi)$ and  $L_{p}(1, \rho) = L_{p}^{*}(1, \rho)$. 
The equality \eqref{eq:p-adic-Stark-at-s=1-eq} for the choice of $\Sigma$ above
now follows by substituting equalities \eqref{eqn:num-omega-expression} and \eqref{eq:denom-omega-expression} into \eqref{eqn:Omega-expression}.
Finally, by Remark \ref{rmk:indep-choice-of-sigma} we obtain the desired result for any choice of $\Sigma$.
\end{proof}

\section{The ETNC and the Equivariant Iwasawa Main Conjecture}

\subsection{Algebraic $K$-theory}
For a left noetherian ring $\Lambda$ we write
$K_{0}(\Lambda)$ for the Grothendieck group of the category of finitely generated 
projective $\Lambda$-modules 
(see \cite[\S 38]{MR892316}) and $K_{1}(\Lambda)$ for the 
Whitehead group (see \cite[\S 40]{MR892316}).
Moreover, we denote the relative algebraic $K$-group associated to a ring homomorphism 
$\Lambda \rightarrow \Lambda'$ by $K_{0}(\Lambda, \Lambda')$.
We recall that $K_{0}(\Lambda, \Lambda')$ is an abelian group with generators $[X,g,Y]$ where
$X$ and $Y$ are finitely generated projective $\Lambda$-modules
and $g:\Lambda' \otimes_{\Lambda} X \rightarrow \Lambda' \otimes_{\Lambda} Y$ 
is an isomorphism of $\Lambda'$-modules;
for a full description in terms of generators and relations, see \cite[p.\ 215]{MR0245634}.
Furthermore, there is a long exact sequence of relative $K$-theory
\begin{equation}\label{eq:long-exact-seq}
K_{1}(\Lambda) \longrightarrow K_{1}(\Lambda') \stackrel{\partial}{\longrightarrow} K_{0}(\Lambda,\Lambda')
\longrightarrow K_{0}(\Lambda) \longrightarrow K_{0}(\Lambda')
\end{equation}
(see  \cite[Chapter 15]{MR0245634}).

Let $R$ be a noetherian integral domain of characteristic $0$, let $E$ be any extension of 
the field of fractions of $R$, and let $G$ be a finite group. We then write
$K_{0}(R[G],E)$ for the relative algebraic $K$-group associated 
to the ring homomorphism $R[G] \hookrightarrow E[G]$ and write $K_{0}(R[G],E)_{\tors}$ for its torsion subgroup.
If $H$ is a subgroup of $G$ then the inclusion map $R[H] \hookrightarrow R[G]$ induces canonical 
restriction and induction maps
\[
\res^{G}_{H}: K_{0}(R[G],E) \longrightarrow K_{0}(R[H],E), \quad
\ind^{G}_{H}: K_{0}(R[H],E) \longrightarrow K_{0}(R[G],E).
\]
Moreover, if $N$ is a normal subgroup of $G$ then the quotient map $R[G] \twoheadrightarrow R[G/N]$
induces a canonical quotient map
\[
\quot^G_{G/N}: K_{0}(R[G],E) \longrightarrow K_{0}(R[G/N],E).
\]
Finally, the maps $K_{0}(\Z[G],\Q) \rightarrow K_{0}(\Z_{p}[G],\Q_{p})$ 
induce a canonical isomorphism
\begin{equation}\label{eq:p-part-decomposition}
K_{0}(\Z[G],\Q) \cong \bigoplus_{p} \, K_{0}(\Z_{p}[G],\Q_{p})
\end{equation}
where $p$ ranges over all primes (see the discussion following \cite[(49.12)]{MR892316}).

\subsection{The equivariant Tamagawa number conjecture (ETNC)}\label{subsec:ETNC-overview}
We give a very brief description of the statement and properties of the equivariant Tamagawa number conjecture (ETNC) for Tate motives formulated by Burns and Flach \cite{MR1884523}; we
omit all details except those necessary for proofs in later sections.

Let $L/K$ be a finite Galois extension of number fields and let $G=\Gal(L/K)$.
For each integer $r$ we set $\Q(r)_{L}:=h^{0}(\mathrm{Spec}(L))(r)$,
which we regard as a motive defined over $K$ and with coefficients in the semisimple algebra $\Q[G]$.
The conjecture `$\ETNC(\Q(r)_{L}, \Z[G])$' formulated in
\cite[Conjecture 4(iv)]{MR1884523}
for the pair $(\Q(r)_{L}, \Z[G])$ asserts
that a certain canonical element $T\Omega(\Q(r)_{L}, \Z[G])$  of $K_{0}(\Z[G],\R)$ vanishes.
(As observed in \cite[\S 1]{MR1981031}, the element $T\Omega(\Q(r)_{L}, \Z[G])$ 
is indeed well-defined.)
We define the following notation.
\begin{itemize}
\item $\ETNC(L/K,r)$ means $T\Omega(\Q(r)_{L}, \Z[G])=0$.
\item $\ETNC^{\tors}(L/K,r)$ means $T\Omega(\Q(r)_{L}, \Z[G]) \in K_{0}(\Z[G],\Q)_{\tors}$.
\item $\ETNC^{\rat}(L/K,r)$ means $T\Omega(\Q(r)_{L}, \Z[G]) \in K_{0}(\Z[G],\Q)$. 
\end{itemize}
Thus if $\ETNC^{\rat}(L/K,r)$ holds then by \eqref{eq:p-part-decomposition} 
we have elements $T\Omega(\Q(r)_{L}, \Z[G])_{p}$ in $K_{0}(\Z_{p}[G],\Q_{p})$ for each prime $p$.
In this situation, we define the following notation.
\begin{itemize}
\item $\ETNC_{p}(L/K,r)$ means $T\Omega(\Q(r)_{L}, \Z[G])_{p}=0$.
\item $\ETNC_{p}^{\tors}(L/K,r)$ means $T\Omega(\Q(r)_{L}, \Z[G])_{p} \in K_{0}(\Z_{p}[G],\Q_{p})_{\tors}$.
\end{itemize}

We now observe that several conjectures encountered thus far are in fact equivalent.

\begin{prop}\label{prop:rationality}
Let $L/K$ be a finite Galois extension of number fields and let $G=\Gal(L/K)$.
Then the following are equivalent:
\begin{enumerate}
\item $\ETNC^{\rat}(L/K,0)$; \tabto{4.5cm} \emph{(ii)} $\ETNC^{\rat}(L/K,1)$;
\setcounter{enumi}{2}
\item Stark's conjecture at $s=0$ for every $\chi \in R_{\C}^+(G)$;
\item Stark's conjecture at $s=1$ for every $\chi \in R_{\C}^+(G)$.
\end{enumerate}
\end{prop}

\begin{proof}
As already observed in Remark \ref{rmk:reformulation-stark-1}, (iii) and (iv) are equivalent by \cite[Chapitre I, Th\'eor\`eme 8.4]{MR782485}.
The equivalence of (i) and (ii) follows from \cite[Theorem 5.2]{MR1884523}.
Finally, \cite[Corollary 1]{MR1981031} gives the equivalence of (i) and (iii).
\end{proof}

\subsection{Reduction steps for the ETNC}\label{subsec:ETNC-reduction-steps}

A Brauer induction argument shows that to prove either $\ETNC^{\rat}(L/K,r)$ or $\ETNC^{\tors}(L/K,r)$
it suffices to consider certain cyclic sub-extensions of $L/K$.
Now fix a prime $p$. Burns \cite[Theorem 4.1]{MR2076565} showed that
\begin{equation}\label{eq:tors-intersection}
K_{0}(\Z_{p}[G],\Q_{p})_{\tors} = \bigcap \ker (\quot^{H}_{Q} \circ \res^{G}_{H}: K_{0}(\Z_{p}[G],\Q_{p}) \longrightarrow K_{0}(\Z_{p}[Q],\Q_{p})), 
\end{equation}
where the intersection runs over all cyclic subgroups $H$ of $G$ and over all quotients $Q$ of $H$ that are of order prime to $p$.
If we assume $\ETNC^{\rat}(L/K,r)$ then \eqref{eq:tors-intersection} and 
the functorial properties of $T\Omega(\Q(r)_L,\Z[G])$ with respect to restriction 
and quotient maps (see \cite[Theorem 4.1]{MR1884523}) together show that to prove
$\ETNC^{\tors}_{p}(L/K,r)$,
it suffices to consider cyclic sub-extensions of $L/K$ of degree prime to $p$.
Moreover, if we assume $\ETNC^{\tors}_{p}(L/K,r)$ then \cite[Proposition 9]{MR1687551} and functoriality show that 
$\ETNC_{p}(L/K,r)$ reduces to certain $p$-elementary Galois sub-extensions of $L/K$
(a finite group is $p$-elementary if it is isomorphic to $C_{m} \times P$ for some $p$-group $P$ and some $m \in \N$ with $p \nmid m$).
Therefore $\ETNC_{p}(L/K,r)$ reduces to the case of abelian (resp.\ cyclic) sub-extensions if $G$ has an abelian
(resp.\ cyclic) Sylow $p$-subgroup.
However, $\ETNC_{p}(L/K,r)$ does not reduce to abelian sub-extensions via functoriality in general.
For example, using the algorithm of Bley and Wilson \cite{MR2564571} (implemented by Bley using \textsc{Magma} \cite{MR1484478}),
one can show that if $G$ is the Heisenberg group of order $27$ then $K_{0}(\Z_{p}[G],\Q_{p})_{\tors}$ has exponent $18$
whereas if $Q$ is any proper subquotient of $G$ then $K_{0}(\Z_{p}[Q],\Q_{p})_{\tors}$ has exponent dividing $6$.
Thus if $Q$ runs over all proper subquotients of $G$ then the map $\prod \quot_{Q}^{H} \circ \res_{H}^{G}$ cannot be injective on
$K_{0}(\Z_{p}[G],\Q_{p})_{\tors}$. 

\subsection{The equivariant Iwasawa main conjecture}\label{subsec:EIMC}
Let $E/F$ be a finite Galois extension of totally real fields and let $G=\Gal(E/F)$.
Fix an odd prime $p$.
Let $E^{\cyc}$ be the cyclotomic $\Z_{p}$-extension of $E$ and let 
$\Sigma$ be a finite set of places of $F$ containing $\Sigma_{\mathrm{ram}}(E/F) \cup \Sigma_{p}(F) \cup \Sigma_{\infty}(F)$ where
 $\Sigma_{\mathrm{ram}}(E/F)$ denotes the set of places of $F$ ramified in $E/F$.
Let $M_{\Sigma}$ be the maximal abelian pro-$p$-extension of $E^{\cyc}$
unramified outside $\Sigma$.
Let $X_{\Sigma}=\Gal(M_{\Sigma}/E^{\cyc})$.
Then $\mathcal{G} := \Gal(E^{\cyc}/F)$ acts on $X_{\Sigma}$ by 
$g \cdot x = \tilde{g}x\tilde{g}^{-1}$ where $g \in \mathcal{G}$ and $\tilde{g}$ is any lift of $g$
to $\Gal(M_{\Sigma}/F)$. Thus $X_{\Sigma}$ is a $\Lambda(\mathcal{G})$-module, where
$\Lambda(\mathcal{G}) := \Z_p \llbracket \mathcal{G} \rrbracket$ denotes the Iwasawa algebra
of $\mathcal{G}$ over $\Z_p$.
We let $\mathcal{Q}(\mathcal{G})$ denote the total ring of fractions of $\Lambda(\mathcal{G})$, 
obtained by adjoining the inverses of all central regular elements.

Since $E$ is totally real, a result of Iwasawa \cite{MR0349627} shows that 
$X_{\Sigma}$ is finitely generated and torsion as a $\Z_{p}\llbracket \Gamma \rrbracket$-module,
where $\Gamma = \Gamma_{E} := \Gal(E^{\cyc}/E) \cong \Z_{p}$. 
We let $\mu_{p}(E)$ denote the Iwasawa $\mu$-invariant of $X_{\Sigma}$ and note that
this does not depend on the choice of $\Sigma$ (see \cite[Corollary 11.3.6]{MR2392026}).
Thus $\mu_{p}(E)=0$ if and only if $X_{\Sigma}$ is finitely generated as a $\Z_{p}$-module.
It is conjectured that we always have $\mu_{p}(E)=0$ and as explained in \cite[Remark 4.3]{MR3749195}, 
this is closely related to the classical Iwasawa `$\mu=0$' conjecture
for $E(\zeta_{p})$ at $p$. Thus a result of Ferrero and Washington \cite{MR528968} on this latter conjecture implies 
that $\mu_{p}(E)=0$ whenever $E/\Q$ is abelian.

We now consider the canonical complex
\[
    C_{\Sigma}^{\bullet}(E^{\cyc}/F) := R\Hom(R\Gamma_{\et}(\Spec(\mathcal{O}_{E^{\cyc}, \Sigma}), \Q_{p} / \Z_{p}), \Q_{p} / \Z_{p}).
\]
Here, $\mathcal{O}_{E^{\cyc}, \Sigma}$ denotes the ring of integers $\mathcal{O}_{E^{\cyc}}$ 
in $E^{\cyc}$ localised at all primes above those in $\Sigma$ and
$\Q_{p} / \Z_{p}$ denotes the constant sheaf of the abelian group $\Q_{p} / \Z_{p}$ 
on the \'{e}tale site
of $\Spec(\mathcal{O}_{E^{\cyc}, \Sigma})$.
The only non-trivial cohomology groups occur in degree $-1$ and $0$ and we have
\[
H^{-1}(C_{\Sigma}^{\bullet}(E^{\cyc}/F)) \cong X_{\Sigma}, \qquad H^{0}(C_{\Sigma}^{\bullet}(E^{\cyc}/F)) \cong \Z_{p}.
\]
It follows from \cite[Proposition 1.6.5]{MR2276851} that $C_{\Sigma}^{\bullet}(E^{\cyc}/F)$ 
is a perfect complex of $\Lambda(\mathcal{G})$-modules.
In particular, $C_{\Sigma}^{\bullet}(E^{\cyc}/F)$ defines a class 
$[C_{\Sigma}^{\bullet}(E^{\cyc}/F)]$ in $K_{0}(\Lambda(\mathcal{G}), \mathcal{Q}(\mathcal{G}))$
(see \cite[\S 2]{MR3068893}, for example).
Note that $C_{\Sigma}^{\bullet}(E^{\cyc}/F)$ and the complex used
by Ritter and Weiss (as constructed in \cite{MR2114937}) become isomorphic in 
the derived category of $\Lambda(\mathcal{G})$-modules by
\cite[Theorem 2.4]{MR3072281} (see also \cite{MR3068897} for more on this topic).
Hence it makes no essential difference which of these complexes we use in the following.

Let $\chi_{\mathrm{cyc}}$ be the $p$-adic cyclotomic character
\[
\chi_{\mathrm{cyc}}: \Gal(E^{\cyc}(\zeta_p)/F) \longrightarrow \Z_{p}^{\times},
\]
defined by $\sigma(\zeta) = \zeta^{\chi_{\mathrm{cyc}}(\sigma)}$ for any $\sigma \in \Gal(E^{\cyc}(\zeta_p)/F)$ 
and any $p$-power root of unity $\zeta$.
Let $\omega$ and $\kappa$ denote the composition of $\chi_{\mathrm{cyc}}$ with the projections 
onto the first and second factors of the canonical decomposition
$\Z_{p}^{\times} = \langle \zeta_{p-1} \rangle \times (1+p\Z_{p})$, respectively;
thus $\omega$ is the Teichm\"{u}ller character.
We note that $\kappa$ factors through both $\mathcal{G}$ and $\Gamma_{F} := \Gal(F^{\cyc}/F)$;
by abuse of notation we also use $\kappa$ to denote the maps with either of these domains.

Now let $\pi_{\rho}: \mathcal{G} \rightarrow \GL_{n}(\mathcal{O})$ be an 
Artin representation (i.e.\ $\pi_{\rho}$ is continuous and has finite image) with character $\rho$, where $\mathcal{O}$ denotes
the ring of integers in some finite extension $L$ of $\Q_{p}$.
Let $\overline{\Q}_{p}$ denote an algebraic closure of $\Q_{p}$.
Choose a topological generator
$\gamma_{F}$ of $\Gamma_{F}$ and put $u := \kappa(\gamma_{F})$.
Each such choice permits the definition of a power series
$G_{\rho, \Sigma}(T) \in \overline{\Q}_{p} \otimes_{\Q_p} Quot(\Z_{p} \llbracket T \rrbracket)$
such that for every $s \in \Z_{p}$ we have
\begin{equation} \label{eq:L_p-power-series}
L_{p, \Sigma}(1-s, \rho) = \frac{G_{\rho, \Sigma}(u^{s}-1)}{H_{\rho}(u^{s}-1)},
\end{equation}
where, for irreducible $\rho$, we have
\[
H_{\rho}(T) = \left\{\begin{array}{ll} \rho(\gamma_{F})(1+T)-1 & \mbox{ if } \Gal(E^{\cyc}/ F^{\cyc}) \subset \ker \rho\\
1 & \mbox{ otherwise.}  \end{array}\right.
\]

We recall the following construction from \cite[\S 3]{MR2217048}. 
For $g \in \mathcal{G}$ we write $\overline{g}$ for its image under the canonical
projection $\mathcal{G} \rightarrow \Gamma_{F}$ and let $\mathcal{Q}^{L}(\Gamma_{F}) := L \otimes_{\Q_{p}} \mathcal{Q}(\Gamma_{F})$.
Each continuous representation $\pi: \mathcal{G} \rightarrow \GL_{n}(\mathcal{O})$ gives rise to a 
continuous group homomorphism $\mathcal{G} \rightarrow \GL_{n}(\mathcal{O} \otimes_{\Z_{p}} \Lambda(\Gamma_{F}))$
defined by $g \mapsto \pi(g) \overline{g}$.
This extends to a ring homomorphism $\mathcal{Q}(\mathcal{G}) \rightarrow M_{n}(\mathcal{Q}^{L}(\Gamma_{F}))$
which in turn induces a homomorphism of abelian groups
\[
	\Phi_{\pi}: K_{1}(\mathcal{Q}(\mathcal{G})) \longrightarrow
	K_{1}(M_{n}(\mathcal{Q}^{L}(\Gamma_{F}))) \cong K_{1}(\mathcal{Q}^L(\Gamma_{F}))
	\cong \mathcal{Q}^L(\Gamma_F)^{\times} \cong Quot(\mathcal{O}\llbracket T \rrbracket)^{\times}.
\]
Here the first isomorphism is induced by Morita equivalence,
the second by taking determinants and the third by using \cite[Lemma 4]{MR2114937} and mapping
$\gamma_{F}$ to $1+T$.  
We define an evaluation map
\begin{eqnarray*}
\phi: Quot(\mathcal{O}\llbracket T \rrbracket) & \rightarrow & L \cup \{\infty\}\\
f(T) & \mapsto & f(0).
\end{eqnarray*}
If $\zeta$ is an element of $K_1(\mathcal{Q}(\mathcal{G}))$ we define
$\zeta(\pi) := \phi(\Phi_{\pi}(\zeta))$.

The following is a formulation of the equivariant Iwasawa main conjecture without its uniqueness statement
(we assume the hypotheses and notation above).

\begin{conj}[equivariant Iwasawa main conjecture] \label{conj:EIMC}
There exists an element $\zeta_{\Sigma} \in K_{1}(\mathcal{Q}(\mathcal{G}))$ such that
$\partial(\zeta_{\Sigma}) = - [C_{\Sigma}^{\bullet}(E^{\cyc}/F)]$ and 
for every irreducible Artin representation $\pi_{\rho}$ of $\mathcal{G}$ with character
$\rho$ and for each integer $r \geq 1$ divisible by $p-1$ we have
\begin{equation}\label{eqn:EIMC-eqn}
  \zeta_{\Sigma}(\pi_{\rho} \kappa^{r}) = L_{p, \Sigma}(1-r, \rho) = j\left(L_{\Sigma}(1-r, \rho^{j^{-1}})\right) 
\end{equation}
for every field isomorphism $j: \C \cong \C_p$.
\end{conj}

\begin{remark}
It can be shown that the validity of Conjecture \ref{conj:EIMC} is independent of the choice of $\Sigma$, 
provided that $\Sigma$ is finite and contains $\Sigma_{\mathrm{ram}}(E/F) \cup \Sigma_{p}(F) \cup \Sigma_{\infty}(F)$.
\end{remark}

The following theorem has been shown by Ritter and Weiss \cite{MR2813337} and by Kakde \cite{MR3091976} independently.

\begin{theorem}\label{thm:EIMC-with-mu}
If $\mu_{p}(E)=0$ then Conjecture \ref{conj:EIMC} holds for $E^{\cyc}/F$.
\end{theorem}

We now set $f_{\rho, \Sigma}(T) := G_{\rho, \Sigma}(T) / H_{\rho}(T)$ and note the following result.

\begin{prop}\label{prop:rho-of-zeta}
Suppose that Conjecture \ref{conj:EIMC} holds for $E^{\cyc}/F$.
Then there exists an element $\zeta_{\Sigma} \in K_{1}(\mathcal{Q}(\mathcal{G}))$ such that
$\partial(\zeta_{\Sigma}) = - [C_{\Sigma}^{\bullet}(E^{\cyc}/F)]$ and 
for every irreducible Artin representation $\pi_{\rho}$ of $\mathcal{G}$ with character
$\rho$ we have
\begin{equation}\label{eqn:rho-of-zeta}
	\Phi_{\pi_{\rho}}(\zeta_{\Sigma}) = f_{\rho, \Sigma}(T). 
\end{equation}
\end{prop}

\begin{proof}
We first make some definitions and observations that are independent of Conjecture \ref{conj:EIMC}.
For $f(T) \in Quot(\mathcal{O}\llbracket T \rrbracket)$ and $a \in\Z_p^{\times}$, we set $(t_{a} f)(T) := f(a(1+T)-1)$. 
Let $r \in \Z$. Then from the definitions and \eqref{eq:L_p-power-series}, we have
\begin{equation}\label{eqn:phi-f-L}
	\phi((t_{u^{r}}f_{\rho, \Sigma})(T)) = f_{\rho, \Sigma}(u^{r}-1) = L_{p, \Sigma}(1-r, \rho). 
\end{equation} 
Let $\pi_{\rho}$ be any irreducible Artin representation of $\mathcal{G}$. 
Then for any $\zeta \in K_{1}(\mathcal{Q}(\mathcal{G}))$
\begin{equation}\label{eqn:Phi-twist-by-power-of-kappa}
	\Phi_{\pi_{\rho}\kappa^r}(\zeta) = t_{u^r}(\Phi_{\pi_{\rho}}(\zeta)). 
\end{equation}

Now assume that Conjecture \ref{conj:EIMC} holds for $E^{\cyc}/F$ and let $\zeta_{\Sigma} \in K_{1}(\mathcal{Q}(\mathcal{G}))$ 
be the element specified in its statement; thus $\partial(\zeta_{\Sigma}) = - [C_{\Sigma}^{\bullet}(E^{\cyc}/F)]$.
Moreover, using \eqref{eqn:EIMC-eqn}, \eqref{eqn:phi-f-L} and \eqref{eqn:Phi-twist-by-power-of-kappa}, 
for every integer $r \geq 1$ divisible by $p-1$ we have
\[
\phi((t_{u^{r}}f_{\rho, \Sigma})(T)) 
= L_{p, \Sigma}(1-r, \rho) 
= \zeta_{\Sigma}(\pi_{\rho} \kappa^{r})
= \phi(\Phi_{\pi_{\rho} \kappa^{r}}(\zeta_{\Sigma}))
= \phi( t_{u^{r}}(\Phi_{\pi_{\rho}}(\zeta_{\Sigma}))).
\]
The desired equality \eqref{eqn:rho-of-zeta} now follows from the $p$-adic Weierstrass preparation theorem (see \cite[Theorem 7.3]{MR1421575}).
\end{proof}

\subsection{The interpolation property at $s=1$}
We keep the notation of the last subsection.
In particular, $\pi_{\rho}$ is an 
irreducible Artin representation of $\mathcal G$ with character $\rho$. 
We denote the order of vanishing of 
$L_{p, \Sigma}(s, \rho)$ at $s=1$ by $r(\rho)$.
Thus by \eqref{eq:p-adic-L-fnc-leading-term}, if $\Leo(E,p)$ holds, we have $r(\rho) = -1$ when $\rho$ is the trivial character and 
$r(\rho)=0$ when $\rho$ is non-trivial. For any $\xi \in K_{1}(\mathcal{Q}(\mathcal{G}))$, let $\xi^{\ast}(\rho)$ denote its leading term at $\rho$ as defined in \cite[\S 2]{MR2749572};
in other words, $\xi^{\ast}(\rho)$ is the leading term at $T=0$ of $\Phi_{\pi_{\rho}}(\xi)$.
The following result is a variant of \cite[Theorem 9.1]{MR3294653}.
(Note that the sign in the exponent of $\log_{p}(u)$ in \eqref{eqn::EIMC-at-1} below is the opposite of that in \cite[Theorem 9.1, (ii) (a)]{MR3294653};
this difference is explained in the proof of Theorem \ref{thm:descent-result} below.)

\begin{prop}\label{prop:EIMC-at-1}
Suppose that Conjecture \ref{conj:EIMC} holds for $E^{\cyc}/F$.
Then there exists an element $\zeta_{\Sigma} \in K_{1}(\mathcal{Q}(\mathcal{G}))$ such that
$\partial(\zeta_{\Sigma}) = - [C_{\Sigma}^{\bullet}(E^{\cyc}/F)]$ and 
for every irreducible Artin character $\rho$ of $\mathcal{G}$ we have
\begin{equation}\label{eqn::EIMC-at-1}
\zeta_{\Sigma}^{\ast}(\rho) = \log_{p}(u)^{-r(\rho)} L_{p, \Sigma}^{\ast}(1, \rho). 
\end{equation}
\end{prop}

\begin{proof}
Let $\zeta_{\Sigma} \in K_{1}(\mathcal{Q}(\mathcal{G}))$ be as in the statement of Proposition \ref{prop:rho-of-zeta}.
Then $\partial(\zeta_{\Sigma}) = - [C_{\Sigma}^{\bullet}(E^{\cyc}/F)]$ and \eqref{eqn:rho-of-zeta} shows that 
$\zeta_{\Sigma}^{\ast}(\rho)$ coincides with the leading term of $f_{\rho, \Sigma}(T)$ at $T=0$.
Now write $f_{\rho, \Sigma}(T) = T^{r(\rho)} \cdot h_{\rho, \Sigma}(T)$
so that $h_{\rho, \Sigma}(0)$ equals the leading term of $f_{\rho, \Sigma}(T)$ at $T = 0$.
Thus $\zeta_{\Sigma}^{\ast}(\rho) = h_{\rho, \Sigma}(0)$.
Moreover, for every $s \in \Z_{p}$
\begin{equation}\label{eqn:Lp-f-rho-h-rho}
L_{p, \Sigma}(1-s, \rho) = f_{\rho, \Sigma}(u^{s}-1) = (u^{s}-1)^{r(\rho)} h_{\rho, \Sigma}(u^{s}-1), 
\end{equation}
where the first equality follows from \eqref{eq:L_p-power-series} and the definition of $f_{\rho, \Sigma}(T)$.

Let $s \in \Z_{p}$.
Since $u = \kappa(\gamma_{F}) \in 1 + p \Z_{p}$ and $p$ is odd, \cite[Proposition 5.7]{MR1421575} gives
\begin{equation}\label{eqn:p-adic-exp-log}
\textstyle{u^{s}-1 = \exp_{p}(\log_{p}(u^{s})) - 1 =  \sum_{n \geq 1} \frac{(\log_{p}(u))^{n}}{n!} s^{n}}. 
\end{equation}
Moreover, 
\begin{eqnarray*}
h_{\rho, \Sigma}(u^{s}-1)
&=& (u^{s}-1)^{-r(\rho)}L_{p, \Sigma}(1-s, \rho)  \quad \textrm{ by \eqref{eqn:Lp-f-rho-h-rho}}\\
&=& \textstyle{\left( \sum_{n \geq 1} \frac{(\log_{p}(u))^{n}}{n!} s^{n} \right)^{-r(\rho)}} L_{p, \Sigma}(1-s, \rho)
\quad \textrm{ by \eqref{eqn:p-adic-exp-log}} \\
&=& \textstyle{\left( \sum_{n \geq 1} \frac{(\log_{p}(u))^{n}}{n!} s^{n-1} \right)^{-r(\rho)}}  s^{-r(\rho)}L_{p, \Sigma}(1-s, \rho).
\end{eqnarray*}
Therefore taking limits as $s \rightarrow 0$ gives
\[
\zeta_{\Sigma}^{\ast}(\rho)
= h_{\rho, \Sigma}(0)
= \lim_{s \rightarrow 0} h_{\rho, \Sigma}(u^{s}-1)
= \log_{p}(u)^{-r(\rho)} L_{p, \Sigma}^{\ast}(1, \rho). \qedhere
\]
\end{proof}

\section{A prime-by-prime descent theorem for the ETNC at $s=1$}

\begin{theorem}\label{thm:descent-result}
Let $E/F$ be a finite Galois extension of totally real number fields and let $G=\Gal(E/F)$.
Fix a prime $p$ and suppose that the $p$-adic Stark conjecture at $s=1$ holds for all $\rho \in \Irr_{\C_{p}}(G)$.
\begin{enumerate}
\item $\ETNC^{\rat}(E/F,1)$ holds.
\item If $p$ is odd then $\ETNC^{\tors}_{p}(E/F,1)$ holds.
\item Suppose $p$ is odd and if $p$ divides $|G|$ then further suppose that $\mu_{p}(E)=0$.\\
Then $\ETNC_{p}(E/F,1)$ holds.
\end{enumerate}
\end{theorem}

\begin{remark}\label{rmk:improvement-compared-to-Burns-descent-result}
Theorem \ref{thm:descent-result} refines Burns' descent result \cite[Corollary 2.8]{MR3294653}.
A key point is that Burns' result assumes the $p$-adic Stark conjecture at $s=1$  for \emph{all} odd primes and then uses
a result on certain maps between relative algebraic $K$-groups \cite[Lemma 2.1]{MR2371375} to deduce $\ETNC^{\rat}(E/F,1)$.
By contrast, we assume the $p$-adic Stark conjecture at $s=1$ for a single prime $p$ and deduce $\ETNC^{\rat}(E/F,1)$ using
Corollary \ref{cor:p-stark-implies-stark}. There are several other differences between the two approaches, but we emphasise that 
both crucially rely on the descent theory of Burns and Venjakob \cite{MR2749572}.
\end{remark}

\begin{proof}[Proof of Theorem \ref{thm:descent-result}]
Claim (i) follows from Corollary \ref{cor:p-stark-implies-stark} and 
Proposition \ref{prop:rationality}. Now assume that $p$ is odd.
By the discussion around \eqref{eq:tors-intersection}, to prove $\ETNC^{\tors}_{p}(E/F,1)$
we may assume without loss of generality that $E/F$ is cyclic of degree prime to $p$.
Thus to prove (ii) it suffices to prove (iii), which we now do.
Let $E^{\cyc}$ denote the cyclotomic $\Z_{p}$-extension of $E$ and let $\mathcal{G}=\Gal(E^{\cyc}/F)$.
Then the equivariant Iwasawa main conjecture for the extension $E^{\cyc}/F$
holds by \cite[Theorem 4.12]{MR3749195} if $p \nmid |G|$ and
by Theorem \ref{thm:EIMC-with-mu} otherwise.
As the $p$-adic Stark conjecture at $s=1$ holds for all $\rho \in \Irr_{\C_{p}}(G)$
by assumption, it follows from Proposition \ref{prop:EIMC-at-1} that
there exists a $\zeta_{\Sigma} \in K_1(\mathcal{Q}(\mathcal{G}))$ such that
$\partial(\zeta_{\Sigma}) = - [C_{\Sigma}^{\bullet}(E^{\cyc}/F)]$ and  we have
\[
\zeta_{\Sigma}^{\ast}(\rho) = \log_{p}(u)^{\langle \rho,\mathbbm{1}_{G} \rangle_{G}} \Omega_j(\rho) \cdot j\left(L_{\Sigma}^{\ast}(1, \rho^{j^{-1}})\right)
\]
for all Artin representations $\rho$ of $\mathcal{G}$ that factor through $G$
and for every $j: \C \cong \C_p$. 
Thus \cite[Theorem 2.2]{MR2749572} implies that \cite[equation (8.8)]{MR2749572} holds.
(To see this, observe that the complex $C_{E^{\cyc}}$ of \cite{MR2749572}
identifies with $C_{\Sigma}^{\bullet}(E^{\cyc}/F)[-3]$ by Artin-Verdier duality;
the shift implies that $[C_{E^{\cyc}}] = - [C_{\Sigma}^{\bullet}(E^{\cyc}/F)]$.
Moreover, in the notation of loc.\ cit., the exponent $-\langle \rho,1 \rangle$
that occurs in \cite[equation (8.8)]{MR2749572}
should in fact be $+\langle \rho,1 \rangle$
because the factor $c_{\gamma}^{\langle \rho, 1\rangle}$
that appears on the left of \cite[(29)]{MR2290587} 
should actually be on the right of that formula.)
Since $T\Omega(\Q(1)_E, \Z[G])$ belongs to $K_{0}(\Z[G],\Q)$ by part (i), 
we can deduce as in the proof of \cite[Theorem 8.1]{MR2749572}
that its image in $K_{0}(\Z_{p}[G],\Q_{p})$ vanishes.
\end{proof}

\begin{remark}
It seems plausible that the result of Theorem \ref{thm:descent-result} (ii) can be deduced using only the Iwasawa main conjecture for totally real fields as
proven by Wiles \cite{MR1053488}, as opposed to the more general equivariant Iwasawa main conjecture.
\end{remark}

\begin{remark}
Let $L/K$ be a finite Galois extension of number fields and fix a prime $p$.
In \S \ref{subsec:ETNC-reduction-steps}, we saw that to prove 
$\ETNC^{\rat}(L/K,1)$, $\ETNC^{\tors}(L/K,1)$ or $\ETNC_{p}^{\tors}(L/K,1)$,
it suffices to consider certain cyclic sub-extensions of $L/K$, but the proof of 
$\ETNC_{p}(L/K,1)$ cannot always be reduced to cyclic (or even abelian) sub-extensions in the same way. 
In the context of Theorem \ref{thm:descent-result}, it is interesting to contrast this with 
the fact that the proof of the $p$-adic Stark conjecture at $s=1$ for $E/F$ a Galois extension 
of totally real number fields always reduces to
certain cyclic sub-extensions of $E/F$
(see Remark \ref{rmk:p-adic-abelian-subextens-reduction}).
\end{remark}

\begin{remark}\label{rmk:vanishing-of-mu-in-Galois-p-extensions}
Let $p$ be an odd prime and let $L/K$ be a finite Galois $p$-extension of number fields.
Then $\mu_{p}(L)=0$ if and only if $\mu_{p}(K)=0$ by \cite[Theorem 11.3.8]{MR2392026}.
Hence the hypothesis that $\mu_{p}(E)=0$ in Theorem \ref{thm:descent-result} (iii) can be weakened to
``there exists a subfield $E'$ of $E$ such that $E/E'$ is a Galois extension of $p$-power degree and $\mu_{p}(E')=0$''.
Moreover, since $\mu_{p}(E)=0$ vanishes whenever $E/\Q$ is abelian (see \S \ref{subsec:EIMC}) we deduce that
$\mu_{p}(E)=0$ whenever $E$ is a Galois $p$-extension of an abelian extension $E'/\Q$.
The same remarks apply to Theorem \ref{thm:LTC-1-from-Leo-for-virtual-perm-chars} and 
Corollaries \ref{cor:rational-frob-inversion} and \ref{cor:LTC-for-every-Galois-group} below.
\end{remark}

\begin{remark}\label{rmk:reasons-for-mu=0-can-weaken-using-hybrid}
There are two reasons for the $\mu_{p}(E) = 0$ hypothesis in Theorem \ref{thm:descent-result}.
The first is that it ensures the validity of the EIMC by Theorem \ref{thm:EIMC-with-mu}.
Using the theory of `hybrid Iwasawa algebras', the present authors \cite{MR3749195}
have proven the EIMC unconditionally in certain cases when it is not known that $\mu_{p}(E)=0$.
Unfortunately, it is not possible to use this result in the present context because the second reason for the
$\mu_{p}(E)=0$ hypothesis is that it is required for
the descent theory of Burns and Venjakob \cite{MR2749572} when $\mathcal{G}$ has an element of order $p$
(i.e.\ when $p$ divides $[E^{\cyc}:F^{\cyc}]$). However, it is still possible to weaken
the $\mu_{p}(E) = 0$ hypothesis in certain situations by using the theory of $p$-adic
hybrid group rings \cite{MR3461042} at the finite level as illustrated by Example \ref{ex:S4-V4-ETNC-at-s=1} below.
\end{remark}

\begin{corollary}\label{cor:absolutely-abelian}
Let $E/F$ be a finite Galois extension of totally real number fields and suppose that
$E/\Q$ is abelian. Then $\ETNC^{\rat}(E/F,1)$ holds and $\ETNC_{p}(E/F,1)$ holds
for every odd prime $p$.
\end{corollary}

\begin{proof}
Let $p$ be a prime. The $p$-adic Stark conjecture at $s=1$ holds by Theorem \ref{thm:p-adic-Stark-for-absolutely-abelian-characters}
and a theorem of Ferrero and Washington \cite{MR528968} implies that $\mu_{p}(E) = 0$ (see \S \ref{subsec:EIMC}).
Therefore $\ETNC^{\rat}(E/F,1)$ and $\ETNC_{p}(E/F,1)$ for $p$ odd follow from Theorem \ref{thm:descent-result}.
\end{proof}

\begin{remark}\label{rmk:comparison-of-approaches-for-ETNC-s=1-in-abs-abelian-case}
A more general version of Corollary \ref{cor:absolutely-abelian}
(without the restrictions that the fields in question are totally real or that $p$ is odd) is certainly well-known, 
but our approach provides a new proof in this particular setting. 
The method of Burns and Flach \cite{MR2290586} uses the validity of $\ETNC(E/F,0)$ 
(as proven outside the $2$-primary part by Burns and Greither \cite{MR1992015} and at $p=2$ by Flach \cite{MR2863902}) 
and  compatibility with the functional equation. 
In some respects, our approach is closer to the proof of Huber and Kings \cite{MR2002643}
of the Bloch-Kato conjecture for Dirichlet characters, which implies (among other results)
$\ETNC_{p}^{\tors}(E/F,1)$ for odd primes $p$; of course, this is somewhat weaker than  $\ETNC_{p}(E/F,1)$.
They formulate a variant of the main conjecture and then descend at $s=1$;
however, they do not prove or refer to the $p$-adic Stark conjecture at $s=1$.
\end{remark}

\begin{corollary}\label{cor:ETNC-at-s=1-perm-chars}
Let $E/F$ be a finite Galois extension of totally real number fields such that $G:=\Gal(E/F)$ satisfies $\Perm(G)=R_{\C}(G)$.
Fix a prime $p$ and suppose that $\Leo(E,p)$ holds. Then conclusions (i), (ii) and (iii) of Theorem \ref{thm:descent-result} hold.
\end{corollary}

\begin{proof}
This is the combination of Corollary \ref{cor:leo-implies-p-adic-Stark-for-virtual-perm-chars} and Theorem \ref{thm:descent-result}. 
\end{proof}

\begin{remark}\label{rmk:all-chars-perm-chars}
The collection of finite groups $G$ such that $\Perm(G)= R_{\C}(G)$ is closed under direct products and includes
the symmetric groups, the hyperoctahedral groups (which include the dihedral group of order $8$), and many others besides.
See \cite{MR0357573} and \cite{MR765700} for more on this topic.
\end{remark}

\begin{example}\label{ex:S4-V4-ETNC-at-s=1}
Let $E$ be a totally real Galois extension of $\Q$ with $\Gal(E/\Q) \cong S_{4}$ and assume that $\Leo(E,3)$ holds.
Then $\ETNC^{\rat}(E/\Q,1)$ and $\ETNC^{\tors}_{3}(E/\Q,1)$ hold by Corollary \ref{cor:ETNC-at-s=1-perm-chars} (i) \& (ii), respectively.
Let $F$ be the subfield of $E$ fixed by $V_{4}$, the subgroup of $S_{4}$ generated by the double transpositions. 
Then $\Gal(F/\Q) \cong S_{4}/V_{4} \cong S_{3}$.
Since $F$ is a cubic Galois extension of a quadratic extension of $\Q$, we have $\mu_{3}(F)=0$
by Remark \ref{rmk:vanishing-of-mu-in-Galois-p-extensions}.
Thus $\ETNC_{3}(F/\Q,1)$ holds by Corollary \ref{cor:ETNC-at-s=1-perm-chars} (iii).
Moreover, the group ring $\Z_{3}[S_{4}]$ is `$V_{4}$-hybrid' by \cite[Example 2.18]{MR3461042}
and so the map 
\[
\quot^{S_{4}}_{S_{4}/V_{4}} : K_{0}(\Z_{3}[S_{4}],\Q_{3})_{\tors} \longrightarrow K_{0}(\Z_{3}[S_{4}/V_{4}],\Q_{3})_{\tors} 
\]
is injective by \cite[Proposition 3.8]{MR3461042}.
Therefore $\ETNC_{3}(E/\Q,1)$ holds by the functoriality properties of the ETNC with respect to quotient maps. 
Note that this result cannot be deduced directly from Corollary \ref{cor:ETNC-at-s=1-perm-chars} (iii)
without additionally assuming that $\mu_{3}(E)=0$,
which illustrates the point made at the end of Remark \ref{rmk:reasons-for-mu=0-can-weaken-using-hybrid}.
\end{example}

\section{The leading term conjectures at $s=0$ and $s=1$}\label{subsec:LTCs}

\subsection{Overview of the leading term conjectures}
We give a brief overview of the leading term conjectures (LTCs) at $s=0$ and $s=1$ formulated by Breuning and Burns;
we refer the reader to their article \cite{MR2371375} and the references therein for further details.

Let $L/K$ be a finite Galois extension of number fields and let $G=\Gal(L/K)$.
For $r \in \{ 0,1 \}$ one can define an element $T\Omega(L/K,r) \in K_{0}(\Z[G],\R)$ that relates leading terms of equivariant Artin $L$-functions at $s=r$ to a certain arithmetic complex. The leading term conjecture at $s=r$ is the assertion that $T\Omega(L/K,r)$ vanishes.
In analogy with \S \ref{subsec:ETNC-overview}, we define the following notation.
\begin{itemize}
\item $\LTC(L/K,r)$ means $T\Omega(L/K,r)=0$.
\item $\LTC^{\tors}(L/K,r)$ means $T\Omega(L/K,r) \in K_{0}(\Z[G],\Q)_{\tors}$.
\item $\LTC^{\rat}(L/K,r)$ means $T\Omega(L/K,r) \in K_{0}(\Z[G],\Q)$. 
\end{itemize}
Thus if $\LTC^{\rat}(L/K,r)$ holds then by \eqref{eq:p-part-decomposition} 
we have elements $T\Omega(L/K,r)_{p}$ in $K_{0}(\Z_{p}[G],\Q_{p})$ for each prime $p$.
In this situation, we define the following notation.
\begin{itemize}
\item $\LTC_{p}(L/K,r)$ means $T\Omega(L/K,r)_{p}=0$.
\item $\LTC_{p}^{\tors}(L/K,r)$ means $T\Omega(L/K,r)_{p} \in K_{0}(\Z_{p}[G],\Q_{p})_{\tors}$.
\item $\LTC_{\odd}(L/K,r)$ means $\LTC_{p}(L/K,r)$ holds for all odd primes $p$.
\end{itemize}

The LTCs have the same functoriality properties as the ETNC and so the
reduction steps for the ETNC described in \S \ref{subsec:ETNC-reduction-steps} also apply to the LTCs.
A brief discussion of the relation of the LTCs to other conjectures is given in \S \ref{sec:introduction}; also see \cite[Propositions 3.6 and 4.4]{MR2371375}.
Finally, we note two important known cases of the LTCs.

\begin{theorem}\label{thm:known-cases-of-LTCs}
Let $L/K$ be a finite Galois extension of number fields. Let $G=\Gal(L/K)$ and let $r \in \{0,1\}$.
\begin{enumerate}
\item If $L/\Q$ is abelian then $\LTC(L/K,r)$ holds.
\item If $\Char_{\Q}(G)= R_{\C}(G)$ then $\LTC^{\tors}(L/K,r)$ holds.
\end{enumerate}
\end{theorem}

\begin{proof}
The first claim is \cite[Corollary 1.3]{MR2804251}, which crucially depends on the results of \cite{MR1992015}, \cite{MR2290586}, 
and \cite{MR2863902}. The second claim is well-known to experts; we give a proof here for the convenience of the reader.
The strong Stark conjecture (as formulated by Chinburg \cite[Conjecture 2.2]{MR724009}) 
for $L/K$ and all $\chi \in R_{\C}^+(G)$ is known to be equivalent to $\LTC^{\tors}(L/K,0)$
by \cite[Proposition 4.4 (ii)]{MR2371375}. However, the strong Stark conjecture holds
for rational valued characters by \cite[Chapitre II, Th\'eor\`eme 6.8]{MR782485}.
Thus $\LTC^{\tors}(L/K,0)$ holds and so does $\LTC^{\tors}(L/K,1)$ by Corollary \ref{cor:LTC-0-1-relations} below.
\end{proof}

\subsection{The epsilon constant conjectures}\label{subsec:ECCs}
The global epsilon constant conjecture formulated by Bley and Burns \cite{MR2005875} is a natural conjecture for
global epsilon constants arising from the compatibility of the leading term conjectures at $s=0$ and $s=1$ with respect to the
functional equation of the equivariant Artin $L$-function. 
Here we consider an equivalent formulation of Breuning and Burns \cite[\S 5]{MR2371375}.
An element $T\Omega^{\loc}(L/K,1) \in K_{0}(\Z[G],\R)$ is defined and the assertion of the conjecture is that this element vanishes.
Moreover, it is known \cite[Theorem 5.2]{MR2371375} that
\begin{equation}\label{eq:TOmega-at-0-and-1}
\psi_{G}^{*}(T\Omega(L/K,0)) - T\Omega(L/K,1) = T\Omega^{\loc}(L/K,1), 
\end{equation}
where $\psi_{G}^{*}$ is a certain involution of $K_{0}(\Z[G],\R)$ (see \cite[\S 2.1.4]{MR2371375}).
Thus if $T\Omega^{\loc}(L/K,1)$ vanishes then $\LTC(L/K,0)$ holds if and only if $\LTC(L/K,1)$ holds.
Furthermore, as shown in \cite[Remark 4.2 (iv)]{MR2005875},
the vanishing of $T\Omega^{\loc}(L/K,1)$ also implies Chinburg's `$\Omega(2)$-conjecture' as formulated in \cite[Question 3.1]{MR786352}.

It is known that we always have $T\Omega^{\loc}(L/K,1) \in K_{0}(\Z[G],\Q)$ and thus \eqref{eq:p-part-decomposition}
defines $p$-parts $T\Omega^{\loc}(L/K,1)_{p} \in K_{0}(\Z_{p}[G],\Q_{p})$ for every prime $p$. 
Breuning \cite{MR2078894} has refined this further by formulating an independent conjecture for each finite Galois extension of $p$-adic fields $N/M$.
He defined an element $R_{N/M} \in K_{0}(\Z_{p}[\Gal(N/M)],\Q_{p})$ incorporating local epsilon constants and conjectured that $R_{N/M}$ always vanishes.
Moreover, this local conjecture is related to the global epsilon constant conjecture by the equation
\begin{equation}\label{eq:lec-gec}
\textstyle{T\Omega^{\loc}(L/K,1)_{p} = \sum_{v \in \Sigma_{p}(K)} \ind^{G}_{G_{w}}(R_{L_{w}/K_{v}}),}
\end{equation}
where $w$ is a fixed place of $L$ above $v$, $G_{w}$ denotes the decomposition group and $\ind^{G}_{G_{w}}$ is the induction map defined
between relative algebraic $K$-groups (see \cite[Theorem 4.1]{MR2078894}). 
In particular, for a fixed prime $p$, the validity of the local conjecture for all non-archimedean completions $L_{w}/K_{v}$
with $v \in \Sigma_{p}(K)$ implies that $T\Omega^{\loc}(L/K,1)_{p}=0$.

\begin{theorem}\label{thm:known-cases-of-ECC}
Let $L/K$ be a finite Galois extension of number fields. Let $G=\Gal(L/K)$ and
let $p$ be a prime. Then $T\Omega^{\loc}(L/K,1)_{p} \in K_{0}(\Z_{p}[G],\Q_{p})_{\tors}$.
Moreover, we have that $T\Omega^{\loc}(L/K,1)_{p}=0$ if for every $v \in \Sigma_{p}(K)$ there is some $w \in 
\Sigma_{p}(L)$ such that $w \mid v$ and at least one of the following holds:
\begin{enumerate}
\item $w$ is at most tamely ramified in $L/K$;
\item $p$ is odd and $L_{w}/\Q_{p}$ is abelian;
\item $p$ is odd and $[L_{w}:\Q_{p}] \leq 15$;
\item $p$ is odd,
$L_{w}/K_{v}$ is abelian and weakly ramified with cyclic ramification group, 
$K_{v}/\Q_{p}$ is unramified and $[K_{v}:\Q_{p}]$ is coprime to the inertia degree of $L_{w}/K_{v}$.
\end{enumerate} 
\end{theorem}

\begin{proof}
By \eqref{eq:lec-gec}, each claim follows from the analogous claim for the local conjecture.
The first claim follows from \cite[Corollary 3.8]{MR2078894}.
Cases (i) and (ii) follow from \cite[Theorem 3.6, Proposition 4.4]{MR2078894}.
Case (iii) follows from \cite[Theorem 1(a)]{MR3073206} and case (iv) follows from \cite[Theorem 1]{MR3520002}.
\end{proof}

\begin{theorem}\label{thm:GEGG-for-C3m:C2-extensions}
Let $m$ be a positive integer and let $G=(C_{3})^{m} \rtimes C_{2}$ where $C_{2}$ acts on $(C_{3})^{m}$ by inversion
(in the case $m=1$ we have $G \cong S_{3}$).
Let $L/K$ be a Galois extension of number fields with $\Gal(L/K) \cong G$
such that $3$ splits completely in $K/\Q$.
Then the global epsilon constant conjecture holds for $L/K$, that is, $T\Omega^{\loc}(L/K,1)=0$.
\end{theorem}

\begin{remark}
Breuning \cite{MR2031413} was the first to show that the global epsilon constant conjecture holds for all $S_{3}$-extensions of $\Q$.
\end{remark}

\begin{proof}[Proof of Theorem \ref{thm:GEGG-for-C3m:C2-extensions}]
We identify $G$ with $\Gal(L/K)$.
By Theorem \ref{thm:known-cases-of-ECC} we have
\[
T\Omega^{\loc}(L/K,1) \in K_{0}(\Z[G],\Q)_{\tors}.
\]
Moreover, $K_{0}(\Z_{2}[G],\Q_{2})_{\tors}$ is trivial by \cite[Lemma 3.10]{MR3461042}.
Hence it suffices to show that $T\Omega^{\loc}(L/K,1)_{3} =0$ in $K_{0}(\Z_{3}[G],\Q_{3})_{\tors}$.
Let $w$ be a fixed place of $L$ above $3$. 
Then by \eqref{eq:lec-gec} and the hypothesis that $3$ splits completely in $K/\Q$,
we are reduced to showing that $R_{L_{w}/\Q_{3}}=0$ in $K_{0}(\Z_{3}[H],\Q_{3})_{\tors}$
where $H=\Gal(L_{w}/\Q_{3})$.
If $H$ is abelian or $H \cong S_{3}$ then we are done by Theorem \ref{thm:known-cases-of-ECC} (ii) or (iii), respectively.
Otherwise, $H \cong (C_{3})^{n} \rtimes C_{2}$ for some $n \geq 2$.
If $n \geq 3$ then $H$ has $(3^{n}-1)/2 \geq 13$ quotients isomorphic to $S_{3}$
(one for each of the quotients of $(C_{3})^{n}$ of order $3$).
However, by \cite[Proposition 4.3]{MR2031413} there are only $6$ Galois extensions $M/\Q_{3}$ with $\Gal(M/\Q_{3}) \cong S_{3}$
(also see \cite{MR2194887}).
Therefore we are now reduced to the case $H=(C_{3})^{2} \rtimes C_{2}$.
From the database of local fields of Jones and Roberts \cite{MR2137362,MR2194887},
we see that there is precisely one Galois extension $M/\Q_{3}$ with $\Gal(M/\Q_{3}) \cong (C_{3})^{2} \rtimes C_{2}$,
namely the Galois closure of the extension with generating polynomial 
$\mathtt{x^{9}+ 3x^{3}+ 9x^{2}+ 9x+ 3}$. 
From the number fields database of Kl\"uners and Malle \cite{MR1901356}, we see that the number field 
with generating polynomial $\mathtt{x^{18} + 45x^{12} + 27x^{6} + 27}$ is a global representative of $M$ of minimal degree.
Applying the algorithm of Bley and Debeerst \cite{MR3073206}  shows that $R_{M/\Q_{3}}=0$.
(The algorithm is implemented in \textsc{Magma} \cite{MR1484478}
with source code bundled in Debeerst's PhD thesis \cite{debeerst-thesis}.)
\end{proof}

\subsection{A common notion of rationality}
We note that all six notions of `rationality' encountered in this article are in fact equivalent.

\begin{prop-def}\label{prop-def:LTC-vs-ETNC-rationality}
Let $L/K$ be a finite Galois extension of number fields and let $G=\Gal(L/K)$.
Then the following are equivalent:
\begin{enumerate}
\item $\LTC^{\rat}(L/K,0)$; 
\tabto{4.5cm} \emph{(ii)} $\LTC^{\rat}(L/K,1)$;
\setcounter{enumi}{2}
\item $\ETNC^{\rat}(L/K,0)$;
\tabto{4.5cm} \emph{(iv)} $\ETNC^{\rat}(L/K,1)$;
\setcounter{enumi}{4}
\item Stark's conjecture at $s=0$ for every $\chi \in R_{\C}^+(G)$;
\item Stark's conjecture at $s=1$ for every $\chi \in R_{\C}^+(G)$.
\end{enumerate}
We denote these equivalent conditions by $\Rat(L/K)$.
\end{prop-def}

\begin{proof}
As $T\Omega^{\loc}(L/K,1)$ belongs to $K_{0}(\Z[G],\Q)$, the equivalence of (i) and (ii)
follows from \eqref{eq:TOmega-at-0-and-1}.
Items (iii)-(vi) are equivalent by Proposition \ref{prop:rationality}.
Finally, \cite[Proposition 3.6]{MR2371375} says that (ii) and (vi) are equivalent.
\end{proof}

\begin{corollary}\label{cor:LTC-0-1-relations}
Let $L/K$ be a finite Galois extension of number fields and suppose that $\Rat(L/K)$ holds. 
Let $G=\Gal(L/K)$ and let $p$ be a prime. 
\begin{enumerate}
\item $\LTC_{p}^{\tors}(L/K,0)$ holds if and only if $\LTC_{p}^{\tors}(L/K,1)$ holds.
\item If $p \nmid |G|$ then the following are equivalent: \\
$\LTC_{p}^{\tors}(L/K,0)$, $\LTC_{p}^{\tors}(L/K,1)$, $\LTC_{p}(L/K,0)$ and $\LTC_{p}(L/K,1)$.
\end{enumerate}
\end{corollary}

\begin{proof}
By Proposition / Definition \ref{prop-def:LTC-vs-ETNC-rationality},
$T\Omega(L/K,r) \in K_{0}(\Z[G],\Q)$ for $r \in \{0,1\}$.
Recall that we always have $T\Omega^{\loc}(L/K,1) \in K_{0}(\Z[G],\Q)$.
Thus we obtain claim (i) by projecting \eqref{eq:TOmega-at-0-and-1} from $K_{0}(\Z[G],\Q)$ to $K_{0}(\Z_{p}[G],\Q_{p})$
and using that $T\Omega^{\loc}(L/K,1)_{p} \in K_{0}(\Z_{p}[G],\Q_{p})_{\tors}$.
Now assume $p \nmid |G|$. Then $\Z_{p}[G]$ is a maximal $\Z_{p}$-order and so $K_{0}(\Z_{p}[G],\Q_{p})_{\tors}$ is trivial
(see \cite[Theorem 2.4 (ii)]{MR2564571}, for instance).
Thus if $r \in \{0,1\}$ then $\LTC_{p}^{\tors}(L/K,r)$ is equivalent to $\LTC_{p}(L/K,r)$, giving claim (ii).
\end{proof}

\subsection{The relation between the leading term conjectures and the ETNC}

\begin{prop}\label{prop:LTC-vs-ETNC-at-0-and-1}
Let $L/K$ be a finite Galois extension of number fields and let $p$ be a prime.
Suppose that $\Rat(L/K)$ holds. Then
\begin{enumerate}
\item $\LTC^{\tors}_{p}(L/K,0)$ holds if and only if $\ETNC^{\tors}_{p}(L/K,0)$ holds;
\item $\LTC_{p}(L/K,0)$ holds if and only if $\ETNC_{p}(L/K,0)$ holds.
\end{enumerate}
Suppose further that \emph{(\textasteriskcentered)} there is a totally complex Galois extension $M$ of $\Q$ such that
$L \subset M$ and $\Leo(M,p)$ holds. Then
\begin{enumerate}
\setcounter{enumi}{2}
\item $\LTC^{\tors}_{p}(L/K,1)$ holds if and only if $\ETNC^{\tors}_{p}(L/K,1)$ holds;
\item $\LTC_{p}(L/K,1)$ holds if and only if $\ETNC_{p}(L/K,1)$ holds.
\end{enumerate}
\end{prop}

\begin{proof}
By \cite[(29)]{MR1981031} the elements $T\Omega(\Q(0)_{L}, \Z[\Gal(L/K)])$ and $T\Omega(L/K,0)$ are equal up to an involution of
 $K_{0}(\Z[\Gal(L/K)], \R)$.
Hence we obtain claims (i) and (ii).

Let $G=\Gal(M/\Q)$, let $H=\Gal(M/K)$ and let $Q=\Gal(L/K)$.
Fix an embedding of fields $j:\R \rightarrow \C_{p}$.
There is a commutative diagram
\[
\xymatrix{
K_{0}(\Z[G],\R) \ar[rr]^{j_{*}^{G}} \ar[d]_{\quot^{H}_{Q} \circ \res_{H}^{G}} & & K_{0}(\Z_{p}[G],\C_{p}) \ar[d]^{\quot^{H}_{Q} \circ \res_{H}^{G}} \\
K_{0}(\Z[Q],\R) \ar[rr]^{j_{*}^{Q}} & & K_{0}(\Z_{p}[Q],\C_{p})
} 
\]
where the horizontal maps are induced by $j$. Moreover, under the assumption (\textasteriskcentered),
the proof of \cite[Proposition 5.1]{MR2804251} shows that $j_{*}^{G}(T\Omega(\Q(1)_{M}, \Z[G]))=j_{*}^{G}(T\Omega(M/\Q,1))$.
Therefore by the functoriality properties of the conjectures with respect to 
restriction and quotient maps, we have
\begin{align*}
j_{*}^{Q}(T\Omega(L/K,1))
&= j_{*}^{Q}(\quot^{H}_{Q} \circ \res_{H}^{G}(T\Omega(M/\Q,1)))\\
&= \quot^{H}_{Q} \circ \res_{H}^{G}(j_{*}^{G}(T\Omega(M/\Q,1)))\\ 
&= \quot^{H}_{Q} \circ \res_{H}^{G}(j_{*}^{G}(T\Omega(\Q(1)_{M}, \Z[G])))\\
&= j_{*}^{Q}(\quot^{H}_{Q} \circ \res_{H}^{G}(T\Omega(\Q(1)_{M}, \Z[G])))\\
&= j_{*}^{Q}(T\Omega(\Q(1)_{L}, \Z[Q])).
\end{align*}
The assumption that $\Rat(L/K)$ holds means that 
both $T\Omega(\Q(1)_{L}, \Z[Q])$ and $T\Omega(L/K,1)$ belong to $K_{0}(\Z[Q],\Q)$.
Furthermore, the map $j_{*}^{Q}$ restricts to the canonical projection $K_{0}(\Z[Q],\Q) \rightarrow K_{0}(\Z_{p}[Q],\Q_{p})$
and so we conclude that $T\Omega(\Q(1)_{L}, \Z[Q])_{p}=T\Omega(L/K,1)_{p}$ in $K_{0}(\Z_{p}[Q],\Q_{p})$.
Hence we obtain claims (iii) and (iv).
\end{proof}

\begin{remark}
\cite[Proposition 5.1]{MR2804251} says that if $M/\Q$ is a finite totally complex Galois extension such that 
$\Leo(M,p)$ holds for \emph{all} primes $p$ then $\LTC(M/\Q,1)$ and $\ETNC(M/\Q,1)$ are equivalent
(no rationality assumption is needed). 
By contrast, Proposition \ref{prop:LTC-vs-ETNC-at-0-and-1} (iii) \& (iv) are useful in proving
`prime-by-prime' results in cases where $\Rat(L/K)$ is known. 
\end{remark}

\begin{remark}\label{rmk:Leo-real-vs-complex}
Let $E/\Q$ be finite totally real Galois extension and let $M$ be a totally complex quadratic extension of $E$. 
Then $\Leo(M,p)$ holds if and only if $\Leo(E,p)$ holds by \cite[Corollary 10.3.11]{MR2392026}.
Moreover, it is straightforward to see that $M$ can be chosen such that $M/\Q$ is Galois.
Therefore if $L=E$ is totally real and $K=\Q$ in Proposition \ref{prop:LTC-vs-ETNC-at-0-and-1}
then one only needs to assume that $\Leo(E,p)$ holds in order to ensure that (\textasteriskcentered) is satisfied. 
It seems plausible that (\textasteriskcentered) could always be weakened to assuming $\Leo(L,p)$,
but this would require a careful generalisation of the proofs of \cite{MR2804251}.
\end{remark}

\section{New evidence for the leading term conjectures}

\subsection{Groups $G$ with the property that $\Perm(G)= R_{\C}(G)$}

\begin{theorem}\label{thm:LTC-1-from-Leo-for-virtual-perm-chars}
Let $E/F$ be a finite Galois extension of totally real number fields such that $G:=\Gal(E/F)$ satisfies $\Perm(G)= R_{\C}(G)$
and $E/\Q$ is Galois. Let $p$ be an odd prime such that $\Leo(E,p)$ holds and $\mu_{p}(E)=0$.
\begin{enumerate}
\item $\LTC^{\tors}(E/F,0)$, $\LTC^{\tors}(E/F,1)$ and  $\LTC_{p}(E/F,1)$ all hold.
\item If $p$ satisfies any of the conditions of Theorem \ref{thm:known-cases-of-ECC} then $\LTC_{p}(E/F,0)$ also holds.
\end{enumerate}
\end{theorem}

\begin{remark}
Recall that the condition $\Perm(G)= R_{\C}(G)$ was discussed in Remark \ref{rmk:all-chars-perm-chars}.
Note that Theorem \ref{thm:LTC-1-from-Leo-for-virtual-perm-chars} is really only interesting in the case that 
$p$ divides $|G|$ (if $p \nmid |G|$ then, as is well-known to experts, Theorem \ref{thm:known-cases-of-LTCs} (ii)
combined with Corollary \ref{cor:LTC-0-1-relations} (ii) gives a more general and unconditional result).
\end{remark}

\begin{proof}[Proof of Theorem \ref{thm:LTC-1-from-Leo-for-virtual-perm-chars}]
Since $\Perm(G)= R_{\C}(G)$ implies that $\Char_{\Q}(G)=R_{\C}(G)$, both $\LTC^{\tors}(E/F,0)$ and $\LTC^{\tors}(E/F,1)$
hold by Theorem \ref{thm:known-cases-of-LTCs} (ii).
Moreover, $\ETNC_{p}(E/F,1)$ holds by Corollary \ref{cor:ETNC-at-s=1-perm-chars} (iii).
Hence Proposition \ref{prop:LTC-vs-ETNC-at-0-and-1} (iv) and Remark \ref{rmk:Leo-real-vs-complex} together show that 
$\LTC_{p}(E/F,1)$ holds. The final claim now follows from Theorem \ref{thm:known-cases-of-ECC}.
\end{proof}

\begin{corollary}\label{cor:rational-frob-inversion}
Let $m$ be a positive integer and let $G = (C_{3})^{m} \rtimes C_{2}$ where $C_{2}$ acts on $(C_{3})^{m}$ by inversion
(in the case $m=1$ we have $G \cong S_{3}$).
Let $E/F$ be a Galois extension of totally real number fields such that $\Gal(E/F) \cong G$ and $E/\Q$ is Galois.
If $\Leo(E,3)$ holds and $\mu_{3}(E)=0$ then $\LTC(E/F,1)$ holds;
if we further suppose either that $3$ splits completely in $F/\Q$ or that
$3$ satisfies any of the conditions of Theorem \ref{thm:known-cases-of-ECC} then $\LTC(E/F,0)$ also holds.
\end{corollary}

\begin{proof}
Let $N$ denote the Sylow $3$-subgroup of $G$.
Note that every character in $\Irr_{\C}(G)$ is lifted from either a linear character of $G/N \cong C_{2}$ or an irreducible degree $2$ 
character of a copy of $C_{3} \rtimes C_{2} \cong S_{3}$ (one for each of the $(3^{m}-1)/2$ quotients of $(C_{3})^{m}$ of order $3$).
Since $\Perm(C_{2})=R_{\C}(C_{2})$ and $\Perm(S_{3})=R_{\C}(S_{3})$, this shows that $\Perm(G)=R_{\C}(G)$.
Thus we can apply Theorem \ref{thm:LTC-1-from-Leo-for-virtual-perm-chars} with $p=3$.
The desired result now follows from Corollary \ref{cor:LTC-0-1-relations} (ii) and
the fact that  $K_{0}(\Z_{2}[G],\Q_{2})_{\tors}$ is trivial (see \cite[Lemma 3.10]{MR3461042});
if $3$ splits completely in $F/\Q$ we also use Theorem \ref{thm:GEGG-for-C3m:C2-extensions}.
\end{proof}

\begin{corollary}\label{cor:LTC-for-every-Galois-group}
Let $G$ be a finite group. There exist infinitely many Galois extensions of totally real number fields $E/F$ with $\Gal(E/F) \cong G$
such that, if $\Leo(E,p)$ holds and $\mu_{p}(E)=0$ for all odd prime divisors $p$ of $|G|$, then for $r \in \{0,1\}$ both
$\LTC^{\tors}(E/F,r)$ and $\LTC_{\odd}(E/F,r)$ hold 
(in fact, if $|G|$ is odd then $\LTC(E/F,r)$ holds).
\end{corollary}

\begin{proof}
By Cayley's theorem, there exists a positive integer $n$ such that $G$ embeds into $S_{n}$, the symmetric group of degree $n$.
Moreover, $\Perm(S_{n})=R_{\C}(S_{n})$ and there are infinitely many 
(totally) real tamely ramified Galois extensions $E/\Q$ with $\Gal(E/\Q) \cong S_{n}$ (see \cite[Proposition 2]{MR1901356}, for example). 
Fix such an extension $E/\Q$.
By Theorem \ref{thm:LTC-1-from-Leo-for-virtual-perm-chars}, $\LTC^{\tors}(E/\Q,0)$ and $\LTC^{\tors}(E/\Q,1)$ both hold and 
$\LTC_{p}(E/\Q,1)$ holds for all odd prime divisors $p$ of $|G|$; moreover, since $E/\Q$ is tamely ramified,
$\LTC_{p}(E/\Q,0)$ also holds for all such primes by Theorem \ref{thm:known-cases-of-ECC} (i) and \eqref{eq:TOmega-at-0-and-1}.
By construction, there is a sub-extension $E/F$ with $\Gal(E/F) \cong G$.
The functoriality properties of the LTCs with respect to restriction together with Corollary \ref{cor:LTC-0-1-relations} (ii)
show that for $r \in \{ 0, 1\}$ both $\LTC^{\tors}(E/F,r)$ and $\LTC_{\odd}(E/F,r)$ hold.
If $|G|$ is odd, then $\LTC_{2}(E/F,r)$ also holds by Corollary \ref{cor:LTC-0-1-relations} (ii).
\end{proof}

\subsection{The group of affine transformations}\label{subsec:affine}

Let $q$ be a prime power and let $\F_{q}$ be the finite field with $q$ elements.
The group $\Aff(q)$ of affine transformations on $\F_{q}$ is the group of transformations
of the form $x \mapsto ax +b$ with $a \in \F_{q}^{\times}$ and $b \in \F_{q}$.
Thus $\Aff(q)$ is isomorphic to the semidirect product $\F_{q} \rtimes \F_{q}^{\times}$ with the natural action.
Note that in particular $\Aff(3) \cong S_{3}$ and $\Aff(4) \cong A_{4}$.

\begin{theorem}\label{thm:p-adic-Stark-for-group-of-affine-transformations}
Let $E/F$ be a Galois extension of totally real number fields such that $G:=\Gal(E/F) \cong \Aff(q)$ for some prime power $q$.
Let $N$ be the commutator subgroup of $G$ and suppose that $E^{N}/\Q$ is abelian (in particular, this is the case when $F=\Q$.) 
Let $p$ be a prime and suppose that $\Leo(E,p)$ holds.
Then the $p$-adic Stark conjecture at $s=1$ holds for every $\rho \in R_{\C_{p}}(G)$.
\end{theorem}

\begin{proof}
The result for linear characters in $R_{\C_{p}}(G)$ follows from Theorem \ref{thm:p-adic-Stark-for-absolutely-abelian-characters}.
We now identify $G$ with $\Aff(q)$ and note that $G$ is a Frobenius group with Frobenius kernel $N=\{ x \mapsto x+b \mid b \in \F_{q} \}$ and
Frobenius complement $H = \{ x \mapsto ax \mid a \in \F_{q}^{\times }\}$ (see \cite[\S 14A]{MR632548} for background on Frobenius groups).
Let $\psi \in \Irr_{\C_{p}}(N)$ with $\psi \neq  \mathbbm{1}_{N}$. 
Then the induced character $\tau := \ind_{N}^{G} \psi$ is of degree $|G/N|=q-1$ and \cite[(14.4)]{MR632548} shows that $\tau \in \Irr_{\C_{p}}(G)$.
Since there are $q-1$ linear characters in $\Irr_{\C_{p}}(G)$ and $(q-1)+(q-1)^{2}=q(q-1)=|G|$,
we conclude that $\tau$ is in fact the unique non-linear character in $\Irr_{\C_{p}}(G)$.
By Frobenius reciprocity (\cite[(10.9)]{MR632548}), Mackey's subgroup theorem (\cite[(10.13)]{MR632548}),
and the fact that $N$ is abelian, we have
\[
\langle \tau, \ind_{H}^{G} \mathbbm{1}_{H} \rangle_{G} = \langle \psi, \res^{G}_{N} \ind_{H}^{G} \mathbbm{1}_{H} \rangle_{N} = \langle \psi, \ind_{\{ e \}}^{N} 
 \mathbbm{1}_{\{ e \}} \rangle_{N} = 1,
\]
where $\{ e \}$ denotes the trivial subgroup of $G$.
Hence $\tau$ can be expressed as $\Z$-linear combination of $\ind_{H}^{G} \mathbbm{1}_{H}$ and linear characters in $\Irr_{\C_{p}}(G)$. 
We have already shown that the $p$-adic Stark conjecture at $s=1$ holds for all linear characters in $\Irr_{\C_{p}}(G)$;
by Corollary \ref{cor:p-Stark-for-induced-chars}, it also holds for $\ind_{H}^{G} \mathbbm{1}_{H}$ since $\Leo(E^{H},p)$ holds.
Thus by Remark \ref{rmk:p-adic-Stark-closed-under-Z-linear-combs} the $p$-adic Stark conjecture at $s=1$ holds for $\tau$ and therefore for all
$\rho \in R_{\C_{p}}(G)$.
\end{proof}

\begin{corollary}\label{cor:LTC-at-s=1-for-group-of-affine-transformations}
Let $E/\Q$ be a totally real Galois extension such that $G := \Gal(E/\Q) \cong \Aff(p^{m})$ for some odd prime $p$ and some positive integer $m$.
If $\Leo(E,p)$ holds then $\LTC(E/\Q,1)$ holds; if we further suppose that 
$p$ satisfies any of the conditions of Theorem \ref{thm:known-cases-of-ECC}, then 
$\LTC(E/\Q,0)$ also holds.
\end{corollary}

\begin{proof}
Initially, we do not assume that $E$ is totally real, that $p$ is odd or that $\Leo(E,p)$ holds.
Let $\ell \neq p$ be a prime.
We observe that $\ETNC^{\tors}(E/\Q,0)$ and $\ETNC_{\ell}(E/\Q,0)$
both hold by \cite[Theorem 4.6]{MR3461042}. 
Then Proposition \ref{prop:LTC-vs-ETNC-at-0-and-1} (ii) implies that
$\LTC^{\tors}(E/\Q,0)$ and $\LTC_{\ell}(E/\Q,0)$ hold as well.
The validity of $\LTC^{\tors}(E/\Q,1)$ follows from Corollary \ref{cor:LTC-0-1-relations} (i).
Let $N$ be the commutator subgroup of $G$.
As the group ring $\Z_{\ell}[G]$ is $N$-hybrid in the sense of \cite[Definition 2.5]{MR3461042}
by \cite[Example 2.16]{MR3461042}, we conclude that $\LTC_{\ell}(E/\Q,1)$ holds
by \cite[Theorem 5.12]{MR3461042}.

Now assume that $p$ is odd, $E$ is totally real and that $\Leo(E,p)$ holds. 
By Theorem \ref{thm:p-adic-Stark-for-group-of-affine-transformations}
the $p$-adic Stark conjecture at $s=1$ holds for every $\rho \in R_{\C_{p}}(G)$.
Moreover, $\mu_{p}(E)=0$ by Remark \ref{rmk:vanishing-of-mu-in-Galois-p-extensions}.
Thus $\ETNC_{p}(E/\Q,1)$ holds by Theorem \ref{thm:descent-result} (iii).
Hence Proposition \ref{prop:LTC-vs-ETNC-at-0-and-1} (iv) and Remark \ref{rmk:Leo-real-vs-complex} together show that 
$\LTC_{p}(E/\Q,1)$ holds. The final claim now follows from Theorem \ref{thm:known-cases-of-ECC}.
\end{proof}

\begin{remark}\label{rmk:on-Affq-corollary}
The first paragraph of the proof of Corollary \ref{cor:LTC-at-s=1-for-group-of-affine-transformations} shows unconditionally that 
for any Galois extension $E/\Q$ with $\Gal(E/\Q) \cong \Aff(p^{m})$ for some prime $p$ and some positive integer $m$,
$\LTC^{\tors}(E/\Q,1)$ holds and $\LTC_{\ell}(E/\Q,1)$ holds for all primes $\ell \neq p$.
Moreover, Corollary \ref{cor:LTC-at-s=1-for-group-of-affine-transformations} can be generalised to the case of extensions $E/F$ as described in the statement of
Theorem \ref{thm:p-adic-Stark-for-group-of-affine-transformations}, subject to the further hypothesis that $E/\Q$ is Galois.
\end{remark}

\subsection{Further specific Galois extensions}\label{subsec:specific-Galois-extensions}

\begin{theorem}\label{thm:rational-frob-inversion-or-D12-extensions-of-Q}
Let $m$ be a positive integer and let $G = (C_{3})^{m} \rtimes C_{2}$ where $C_{2}$ acts on $(C_{3})^{m}$ by inversion
(in the case $m=1$ we have $G \cong S_{3}$) or let $G=D_{12}$ (the dihedral group of order $12$).
Let $E/\Q$ be a totally real Galois extension with $\Gal(E/\Q) \cong G$.
If $\Leo(E,3)$ holds then both $\LTC(E/\Q,0)$ and $\LTC(E/\Q,1)$ hold.
\end{theorem}

\begin{remark}\label{rmk:Leopoldt-compute}
For a particular extension $E/\Q$ satisfying the hypotheses of Theorem \ref{thm:rational-frob-inversion-or-D12-extensions-of-Q}
one can computationally verify $\Leo(E,3)$ as described in the Appendix. 
In this particular setting, this is a simpler way of verifying $\LTC(E/\Q,0)$ and $\LTC(E/\Q,1)$ than the algorithms
of Janssen \cite{janssen-thesis} and Debeerst \cite{debeerst-thesis}, respectively.
Similar remarks also apply to Corollary \ref{cor:LTC-at-s=1-for-group-of-affine-transformations}.
We note that reducing the problem to the verification of Leopoldt's conjecture for a single prime $p$ is crucial for this computational approach.
\end{remark}

\begin{proof}[Proof of Theorem \ref{thm:rational-frob-inversion-or-D12-extensions-of-Q}]
In each case, $E$ is a Galois $3$-extension of a quadratic or biquadratic extension of $\Q$ and so $\mu_{3}(E)=0$
by Remark \ref{rmk:vanishing-of-mu-in-Galois-p-extensions}. 
If $G = (C_{3})^{m} \rtimes C_{2}$ then
$\LTC(E/\Q,0)$ and $\LTC(E/\Q,1)$ both hold by Corollary \ref{cor:rational-frob-inversion}.
Now suppose that $G = D_{12}$.
Since $D_{12} \cong S_{3} \times C_{2}$ we see that $\Perm(D_{12})=R_{\C}(D_{12})$
and so $\LTC^{\tors}(E/\Q,0)$, $\LTC^{\tors}(E/\Q,1)$ and $\LTC_{3}(E/\Q,1)$ all hold by 
Theorem \ref{thm:LTC-1-from-Leo-for-virtual-perm-chars} (i).
Moreover, by \cite[Example 3.13]{MR3461042} the group ring $\Z_{2}[D_{12}]$ is weakly $N$-hybrid
where $N$ is the Sylow $3$-subgroup of $D_{12}$, meaning that the map
\begin{equation}\label{ex:D12-weak-hynrid-quotient-map}
\quot^{D_{12}}_{D_{12}/N} : K_{0}(\Z_{2}[D_{12}],\Q_{2})_{\tors} \longrightarrow K_{0}(\Z_{2}[D_{12}/N],\Q_{2})_{\tors} 
\end{equation}
is injective. 
Thus $\LTC_{2}(E/\Q,1)$ holds by Theorem \ref{thm:known-cases-of-LTCs} (i) and 
the functoriality properties of the LTCs with respect to quotients.
Therefore $\LTC(E/\Q,1)$ holds by Corollary \ref{cor:LTC-0-1-relations} (ii) and 
so $\LTC(E/\Q,0)$ also holds by Theorem \ref{thm:known-cases-of-ECC} (iii).
\end{proof}

\begin{corollary}\label{cor:S4-S4xC2-extensions-of-Q}
Let $L/\Q$ be a Galois extension with $\Gal(L/\Q) \cong S_{4}$ or $S_{4} \times C_{2}$.
If $E:=L^{V_{4}}$ is totally real and $\Leo(E,3)$ holds then
$\LTC_{\odd}(L/\Q,0)$ and $\LTC_{\odd}(L/\Q,1)$ both hold.
\end{corollary}

\begin{proof}
Let $G=S_{4}$ or $S_{4} \times C_{2}$ and let $r \in \{0,1\}$. 
Since $\Perm(G)=R_{\C}(G)$, we have that $\LTC^{\tors}(L/\Q,r)$ holds
by Theorem \ref{thm:known-cases-of-LTCs} (ii).
Since $G/V_{4}$ is isomorphic to either $S_{3}$ or $S_{3} \times C_{2} \cong D_{12}$,
Theorem \ref{thm:rational-frob-inversion-or-D12-extensions-of-Q} shows that $\LTC(E/\Q,r)$ holds. 
Moreover, the group ring $\Z_{3}[S_{4}]$ is `$V_{4}$-hybrid' by \cite[Example 2.18]{MR3461042}
and so $\Z_{3}[S_{4} \times C_{2}]$ is also $V_{4}$-hybrid by \cite[Lemma 2.9]{MR3461042}.
Hence
\[
\quot^{G}_{G/V_{4}} : K_{0}(\Z_{3}[G],\Q_{3})_{\tors} \longrightarrow K_{0}(\Z_{3}[G/V_{4}],\Q_{3})_{\tors} 
\]
is injective by \cite[Proposition 3.8]{MR3461042}.
Thus by the functoriality properties of the LTCs with respect to quotient maps, $\LTC_{3}(L/\Q,r)$ also holds.
Therefore $\LTC_{\odd}(L/\Q,r)$ holds by Corollary \ref{cor:LTC-0-1-relations}.
\end{proof}

\begin{remark}
It is possible to have $L$ totally complex and $E$ totally real in Corollary \ref{cor:S4-S4xC2-extensions-of-Q}
(consider the Galois closure of $\mathtt{x^{4}+x+2}$).
Moreover, it is interesting to compare Theorem \ref{thm:rational-frob-inversion-or-D12-extensions-of-Q} and Corollary \ref{cor:S4-S4xC2-extensions-of-Q} to 
\cite[Theorem 4.18]{MR3461042} and to Example \ref{ex:S4-V4-ETNC-at-s=1}.
\end{remark}

\newpage

\appendix

\section{Computational verification of Leopoldt's conjecture and the Leading Term Conjectures for totally real fields\\
by Tommy Hofmann, Henri Johnston and Andreas Nickel}\label{appendix}

\subsection{An algorithm for verifying Leopoldt's conjecture}\label{subsec:Leo-alg}
For a comprehensive discussion of Leopoldt's conjecture, we refer the reader to \cite[Chapter X, \S 3]{MR2392026}.
In \cite{MR904010},
Buchmann and Sands described an algorithm to verify Leopoldt's conjecture $\Leo(K,p)$ for a given number field $K$ and prime $p$.
Unfortunately, it appears that there is no currently available implementation of this algorithm.
Here we describe a more direct approach that verifies $\Leo(E,p)$ for a given totally real finite Galois extension $E/\Q$ and prime $p$
by taking advantage of the features of a modern computer algebra system.

Fix a prime $p$, let $\iota \colon \C \cong \C_{p}$ be any field isomorphism and let $\log_{p}: \C_{p}^{\times} \rightarrow \C_{p}$
denote the $p$-adic logarithm. Let $E/\Q$ be a totally real finite Galois extension.
Let $n = [E : \Q]$, let $\sigma_{1},\dotsc,\sigma_{n-1}$ be distinct embeddings of $E$ into $\R$ and let $\varepsilon_{1},\dotsc,\varepsilon_{n-1}$
be a system of fundamental units in $\mathcal{O}_{E}^{\times}$. 
Then the $p$-adic regulator of $E$ is
\[
R_{E,p}:= \det(\log_{p} \iota \circ \sigma_{i}(\varepsilon_{j}))_{1 \leq i, j \leq {n-1}}
\]
and this is well-defined up to sign. Moreover, $\Leo(E,p)$ holds if and only if $R_{E,p} \neq 0$.

Now suppose that $\eta_{1},\dotsc, \eta_{n-1}$ are independent units in $\mathcal{O}_{E}^{\times}$.
Then there exists a matrix $A = (a_{ij})_{1 \leq i, j\leq n-1} \in \operatorname{GL}_{n-1}(\Q)$ with 
$\eta_{i} = \prod_{1 \leq j \leq n-1}\varepsilon_{j}^{a_{ij}}$ for $1 \leq i \leq n-1$ and we have
\[
\det(\log_{p} \iota \circ \sigma_{i}(\eta_{j}))_{1 \leq i, j \leq {n-1}} = \det(A) R_{F,p}.
\]
Therefore  $\Leo(E, p)$ holds if and only if
\begin{align}\label{eq:leocond}
\det(\log_{p} \iota \circ \sigma_{i}(\eta_{j}))_{1 \leq i, j \leq {n-1}} \neq 0.
\end{align}
Let $w$ be a place of $E$ above $p$ and let $E_{w}$ denote the completion of $E$ at $w$.
Abusing notation, we henceforth let $\log_{p} : E_{w}^{\times} \rightarrow E_{w}$ denote the restriction of $\log_{p} : \C_{p}^{\times} \rightarrow \C_{p}$.
Then~(\ref{eq:leocond}) and thus $\Leo(E,p)$ are equivalent to
\begin{align}\label{eq:leocond2}
d := \det(\log_{p}{} \circ \tilde \sigma_{i}(\eta_{j}))_{1 \leq i, j \leq {n-1}} \neq 0,
\end{align}
where $\tilde \sigma_{1},\dotsc,\tilde \sigma_{n} \in \Gal(E/\Q)$.
Let $v_{w}$ denote the normalised $w$-adic valuation on $E_{w}$.
Using the power series expansion of $\log_{p}$, for every $\alpha \in E^{\times}$ and $N \geq 1$ we can find $\beta \in E$
such that $v_{w}(\log_{p}(\alpha) - \beta) \geq N$ (also see \cite[\S 3.1]{MR3506905}).

To verify $\Leo(E, p)$, we can now proceed as follows.
\begin{enumerate}
\item Compute the field automorphisms $\tilde \sigma_{1}, \ldots, \tilde \sigma_{n}$ in $\Gal(E/\Q)$.
\item Compute a set $\eta_{1},\dotsc,\eta_{n-1} \in \mathcal{O}_{E}^{\times}$ of independent units.
\item Let $N = 1$.
\item Compute $R_{N} = (\beta_{ij})_{1 \leq i,j \leq n -1} \in M_{n-1}(E)$ such that
        $v_{w}(\log_{p}(\sigma_{i}(\eta_{j})) - \beta_{ij}) \geq N$ for all $1 \leq i,j\leq n -1$.
\item Compute $d_{N} = \det(R_{N}) \in E$. If $v_{w}(d_{N}) < N$ then $\Leo(E, p)$ holds and we stop. Otherwise
        replace $N$ by $N + 1$ and go to step (iv).
\end{enumerate}
It remains to show that the procedure terminates if and only if $\Leo(E, p)$ holds.
Let $d$ be defined as in \eqref{eq:leocond2}.
We have
\[
v_{w}(d) \geq \min\{ v_{w}(d_{N}), v_{w}(d-d_{N}) \},
\]
with equality if and only if $ v_{w}(d_{N}) \neq v_{w}(d-d_{N})$.
Moreover, by construction we have $v_{w}(d - d_{N}) \geq N$ for all $N \in \Z_{\geq 1}$.
Thus $v_{w}(d_{N}) < N$ if and only if $v_{w}(d) < N$. 
Therefore there exists $N \in \Z_{\geq 1}$ such that $v_{w}(d_{N}) < N$ if and only if $d \neq 0$, which in turn is equivalent to 
$\Leo(E,p)$.

\begin{remark}
One way of finding independent units is to follow the classical class group algorithm of
Buchmann~\cite{MR1104698}, but skip all verification steps. We give a brief description of this method.
After choosing a non-empty set $S$ of non-zero prime ideals of $\mathcal{O}_{E}$, we consider $\varphi \colon
\Z^{\lvert S \rvert} \to \Cl_E,\ (e_\mathfrak p)_{\mathfrak p \in S} \mapsto
[ \prod_{\mathfrak p \in S} \mathfrak p^{e_\mathfrak p}]$. Next we choose $l > \lvert S \rvert$ and randomly compute
$\alpha_1,\dotsc,\alpha_l \in E^\times$ and $(v_i)_{1 \leq i \leq l} = ((e_{i, \mathfrak p})_\mathfrak p)_{1
\leq i \leq l}$ such that $\alpha_i \mathcal O_E = \prod_{\mathfrak p \in S}\mathfrak p^{e_{i,\mathfrak p}}$ for $1 \leq i 
\leq l$. We then consider $\psi \colon \Z^l \to \Z^{\lvert S \rvert}, \ (\mu_i)_i \mapsto \sum_{i}\mu_i v_i$ and note
that any kernel element $(w_i)_{1 \leq i \leq l} \in \ker(\psi)$ yields a unit $\prod_j \alpha_j^{w_j} \in \mathcal{O}_{E}^\times$. 
This is repeated until $1,\dotsc,n-1$ independent units are found.
\end{remark}

This algorithm has been implemented in \textsc{Magma} \cite{MR1484478} by the first-named author of the Appendix and the code is available on his website. 

\subsection{Enumeration of number fields and computational results}

\begin{theorem}\label{thm:leo-holds-at-3}
Let $E/\Q$ be a totally real finite Galois extension. If either
\begin{enumerate}
\item $\Gal(E/\Q) \cong S_{3}$ and $\Disc(\mathcal{O}_{E}) < 10^{20}$, or
\item $\Gal(E/\Q) \cong D_{12}$ and $\Disc(\mathcal{O}_{E}) < 10^{30}$,
\end{enumerate}
then $\Leo(E, 3)$ holds.
\end{theorem}

\begin{proof}
For a given totally real field $E$ satisfying (i) or (ii), we verified $\Leo(E, 3)$ using the implementation of the algorithm described in \S \ref{subsec:Leo-alg}.
It thus remains to describe how complete tables of the totally real fields satisfying (i) and (ii) were computed.

(i) Building on work of Belabas~\cite{MR1415795}, Cohen and Thorne \cite{MR3215550} used analytic methods to
determine explicit counts of totally real $S_{3}$-extensions $E/\Q$ satisfying $\Disc(\mathcal{O}_{E})< X$ for various values of
$X$, including the count of 492\,335 for $X = 10^{20}$.
This method is unconditional, but does not give defining equations for the $S_{3}$-extensions.
To find these equations, we used the class field theoretic algorithms from~\cite{FHS2018} to construct $S_{3}$-extensions of $\Q$ as $C_{3}$-extensions of real quadratic fields.
Although some of the subcomputations (class and ray class group computations, for example) rely on the generalised Riemann hypotheses,
by checking with the unconditional count of $492\,335$, we obtain an unconditional complete
list of totally real $S_{3}$-extensions $E$ with $\Disc(\mathcal{O}_{E}) < 10^{20}$.

(ii) To compute the table of totally real $D_{12}$-extensions of $\Q$, we exploited the isomorphism $D_{12} \cong C_{2} \times S_{3}$.
Using the transitivity of discriminants in towers of number fields, we see that any totally real $D_{12}$-extension $E/\Q$ with $\Disc(\mathcal{O}_{E}) < X$
is the compositum of a real quadratic field $K$ with $\Disc(\mathcal{O}_{K}) < X^{1/6}$ and a totally real
$S_{3}$-extension $L/\Q$ with $\Disc(\mathcal{O}_{L}) < X^{1/2}$.
Since for $X = 10^{30}$ we have $X^{1/2} = 10^{15} < 10^{20}$, the complete table of totally real $S_{3}$-fields from part~(i) is sufficient to
obtain the complete table of totally real $D_{12}$-extensions $E$ with $\Disc(\mathcal{O}_{E}) < 10^{30}$ unconditionally.
There are $24\,283$ such fields.

In both cases the computation of the fields was carried out using \textsc{Hecke}~\cite{MR3703682}.
\end{proof}

\begin{corollary}
For a totally real Galois extension $E/\Q$ satisfying either \emph{(i)} or \emph{(ii)} in Theorem \ref{thm:leo-holds-at-3} 
both $\LTC(E/\Q,0)$ and $\LTC(E/\Q,1)$ hold. 
\end{corollary}

\begin{proof}
This is the combination of Theorems \ref{thm:rational-frob-inversion-or-D12-extensions-of-Q} and \ref{thm:leo-holds-at-3}.
\end{proof}

\subsection*{Acknowledgements}
The first named author acknowledges financial support provided by the DFG 
within Project II.2 of SFB-TRR 195  `Symbolic Tools in Mathematics and their Applications'. 
The financial acknowledgements of the other two authors are as for the main article.
The authors are grateful to Jonathan Sands for useful correspondence regarding \cite{MR904010}, and to MathOverflow user `znt' and 
David Loeffler for the initial sketch of the idea for the algorithm to verify Leopoldt's conjecture. 

\bibliography{p-adic-Stark-bib}{}
\bibliographystyle{amsalpha}

\end{document}